\documentclass[11pt,oneside]{amsart}
\usepackage{amssymb, amsmath, amsthm, stmaryrd}

\usepackage{epsfig}
\usepackage{graphicx}
\usepackage{color}
\usepackage{enumerate}

\usepackage{amsrefs}
\usepackage{fullpage}
\usepackage{pinlabel}
\usepackage{hyperref}

\theoremstyle{plain}
\newtheorem{theorem}{Theorem}[section]

\newtheorem*{theorem*}{Theorem}

\newtheorem{proposition}[theorem]{Proposition}
\newtheorem{corollary}[theorem]{Corollary}
\newtheorem{lemma}[theorem]{Lemma}

\theoremstyle{definition}
\newtheorem{remark}[theorem]{Remark}

\newtheorem{definition}[theorem]{Definition}

\theoremstyle{definition}

\newcommand{\R}{{\mathbb R}}

\newcommand{\Z}{\mathbb Z}
\newcommand{\N}{\mathbb N}

\renewcommand{\H}{\mathbb H}
\newcommand{\T}{\mathbb T}

\newcommand{\nil}{\varnothing}

\newcommand{\wihat}{\widehat}
\newcommand{\defn}[1]{\emph{#1}}
\newcommand{\boundary}{\partial}
\newcommand{\mc}[1]{\mathcal{#1}}
\newcommand{\ob}[1]{\overline{#1}}
\newcommand{\inter}[1]{\mathring{#1}} 

\newcommand{\width}{\operatorname{w}}
\newcommand{\tr}{\operatorname{trunk}}
\newcommand{\val}{\operatorname{val}}

\newcommand{\cpt}{\sqsubset}
\newcommand{\extent}{\operatorname{extent}}
\newcommand{\netextent}{\operatorname{net extent}}
\newcommand{\up}{\uparrow}
\newcommand{\down}{\downarrow}

\renewcommand{\b}{\mathfrak{b}}

\newcommand{\co}{\mskip0.5mu\colon\thinspace}

\title{Links, bridge number, and width trees}
\author{Qidong He and Scott A. Taylor}

\begin{document}
\maketitle

\begin{abstract}
To each link $L$ in $S^3$ we associate a collection of certain labelled directed trees, called width trees. We interpret some classical and new topological link invariants in terms of these width trees and show how the geometric structure of the width trees can bound the values of these invariants from below. We also show that each width tree is associated with a knot in $S^3$ and that if it also meets a high enough ``distance threshold'' it is, up to a certain equivalence, the unique width tree realizing the invariants.
\end{abstract}

\section{Introduction}

The bridge number $\b(L)$ of a link $L\subset S^3$ is a classical invariant with several equivalent definitions. One of best known says that $\b(L)$ can be calculated by positioning $L$ in $S^3$ so that, with respect to some height function, the only critical points on $L$ are isolated maxima and minima. The invariant $\b(L)$ is then the minimum number of maxima over all possible positions and height functions. 

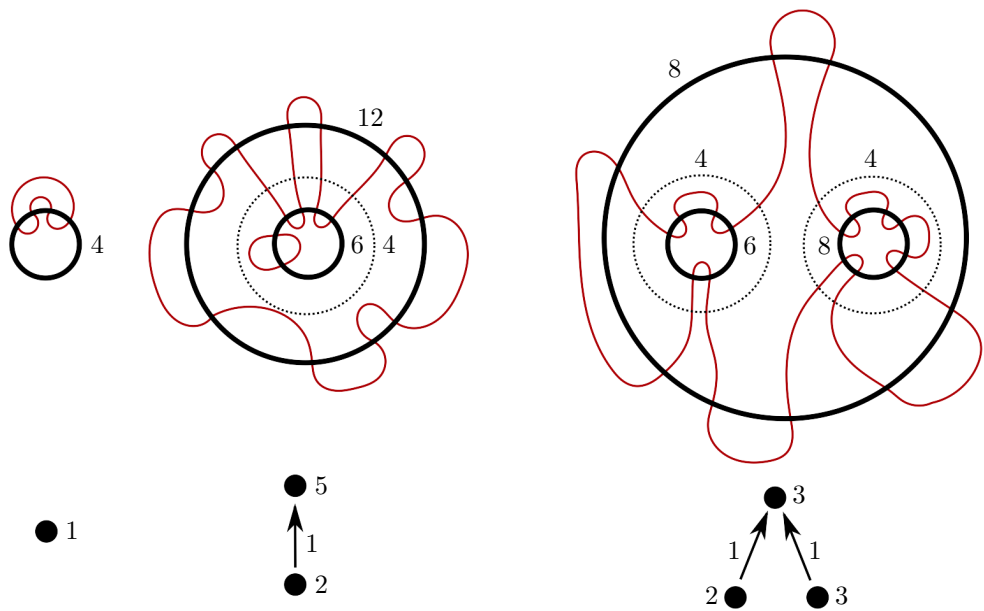
\begin{figure}[ht!]
\centering
\labellist
\small\hair 2pt
\pinlabel {$4$} [l] at 57 275
\pinlabel {$1$} [l] at 38 59
\pinlabel {$6$} [l] at 253 275
\pinlabel {$4$} [l] at 277 275
\pinlabel {$12$} [bl] at 259 361
\pinlabel {$5$} [l] at  226 93
\pinlabel {$1$} [l] at 219 49
\pinlabel {$2$} [l] at 226 18
\pinlabel {$8$} [br] at 509 398
\pinlabel {$4$} [b] at  521 329
\pinlabel {$4$} [b] at 649 329
\pinlabel {$6$} [l] at 549 274
\pinlabel {$8$} [r] at  623 274
\pinlabel {$3$} [l] at 586 86
\pinlabel {$2$} [r] at 538 10
\pinlabel {$3$} [l] at 618 10
\pinlabel {$1$} [r] at 555 44
\pinlabel {$1$} [l] at 598 44
\endlabellist
\includegraphics[scale=0.5]{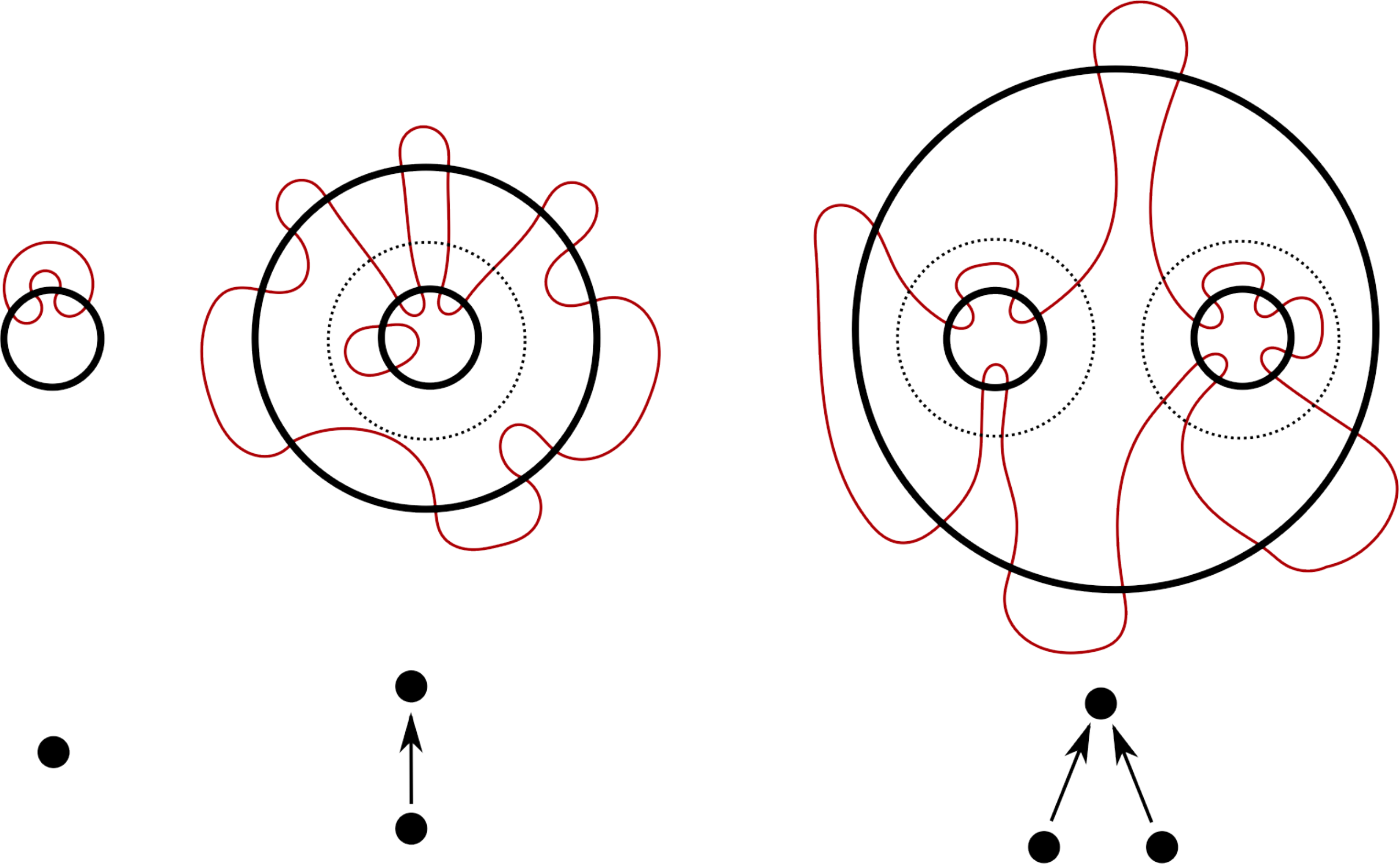}
\caption{Examples of width trees with few edges. Each diagram represents an infinite family of links by inserting braids along each sphere. The vertices and edges of the graph correspond to the spheres. We have labelled each sphere with the number of points of intersections between it and the link. We have labelled each vertex or edge by $|S \cap L|/2 - 1$ where $S$ is the sphere corresponding to the vertex or edge.}
\label{fig:smallex}
\end{figure}

All links $L$ with $\b(L) = 2$ have the form shown on the left of Figure \ref{fig:smallex}. The height function is the distance from the center of the ball shown. For that height function, all the maxima are above all the minima. The sphere $H$ that is the boundary of the ball separates the maxima from the minima and is called a bridge sphere for $L$. Notice that the number of maxima is exactly half the number of times $L$ intersects $H$. However, as in the center of Figure \ref{fig:smallex}, not all maxima need be above all minima. In that example, we have three spheres of interest. There are two ``thick spheres'' $H_1, H_2$ that seperate maxima from minima, with maxima above minima. There is one ``thin sphere'' $F$ that separates maxima from minima, with minima above maxima. The thin and thick spheres are all concentric. Notice that the number of maxima is exactly equal to the ``alternating sum'' $\frac{1}{2}\big(|H_1 \cap L| + |H_2 \cap L| - |F \cap L|\big)$. Finally, on the right of Figure \ref{fig:smallex}, observe that we can again identify thin and thick spheres, but this time, they are not concentric. This time, the height function takes values in a tree, but again the number of maxima can be calculated from the intersections with the thick and thin surfaces.

In all three instances, we can keep track of the relevant counts by considering the dual tree to the thin spheres. Each vertex corresponds to a thick sphere and each edge to a thin sphere. The height functions induce orientations on the edges of these dual trees. If a vertex or edge corresponds to a thick or thin sphere $S$, we label it with $|S \cap L|/2 - 1$. Delaying a precise definition until later, in each case we have constructed what we call a ``width tree'' for the link $L$. Figure \ref{fig:widthtree} shows another example of a width tree; it is associated to links that can be presented in the form shown in Figure \ref{fig:multicbridge}. That example differs from those in Figure \ref{fig:smallex} most significantly in the existence of what we call a ``ghost edge.'' We will say more about those later. It turns out that $\b(L)$ can be computed by minimizing the ``alternating sum'' over all possible width trees for $L$. From the set of width trees, we can compute other invariants, including Gabai's width $\width_G(L)$ and newer invariants due to Taylor--Tomova. The geometry of the width trees affects the values of these invariants. Deferring some definitions and stating our results informally, we prove the following:
\begin{itemize}
\item Associated to each even tangle $(M,\tau)$, including each knot and link in the 3--sphere $S^3$, are collections of width trees $\T(M, \tau)$ and $\T_2(M,\tau) \subset \T(M,\tau)$. Bridge number, Gabai width, and Taylor-Tomova's invariants can be computed by minimizing over elements of $\T_2(M,\tau)$ (Section \ref{mult cbridge});
\item The width trees in $\T_2(M,\tau)$ determine whether or not $(M,\tau)$ is a split or composite tangle and can also reflect the existence of essential Conway spheres (Theorem \ref{thm: slim main});
\item The geometry of the width trees in $\T(M,\tau)$ bounds bridge number, width, and their difference, from below (Theorem \ref{minflowmaxcut} and Corollary \ref{cor: difference});
\item Every positive, productless width tree lies in $\T_2(M,\tau)$ for some even tangle $(M,\tau)$ (Theorem \ref{creation});
\item If the width tree is also equipped with a ``distance threshold,'' and if the distance threshold is large enough, then there is a unique (up to equivalence) width tree realizing both bridge number and width (Corollary \ref{uniqueness});
\item Consequently, the lower bound we exhibit for bridge number is sharp (Corollary \ref{cor: sharp}).
\end{itemize}
We will explain each of the terms subsequently, but we start by reviewing the classical invariants of bridge number and Gabai width. The perspective we bring to them is then generalized to produce our width trees. Our work is influenced by and based heavily on the work of Hayashi--Shimokawa \cite{HS},  Tomova \cite{Tomova2}, Blair--Tomova \cite{BT}, Blair--Zupan \cite{BZ}, and Taylor--Tomova \cites{TT1, TT2} and is a graph theoretic reformulation of the generalized bridge positions studied by those authors. This reformulation produces additional insight into bridge number and suggests ways to construct examples showing that certain invariants are nonadditive under connected sum. It is also intended to further justify the view that Taylor and Tomova's invariants ``net extent'' and ``width'' are useful extensions of the classical invariants ``bridge number'' and ``Gabai width.''

In the remainder of this section, we elaborate on the preceding discussion with more detail. Suppose that $L \subset S^3$ is a link. A \defn{height function} for $L$ is a smooth function $h \co S^3 \to I$ (where $I = [-1,+1] \subset \R$) such that $h$ has exactly two critical points, with critical values $\pm 1$,  and $h|_L$ is Morse with critical points at distinct heights, none of which are at $\pm 1$. A \defn{bridge position} for $L$ is a height function $h \co S^3 \to I$ such that all minima are below all maxima. The \defn{bridge number} $\b(h)$ of a bridge position $h$ is defined to be the number of maxima. The \defn{bridge number} $\b(L)$ of $L$ is the minimum bridge number across all bridge positions for $L$. Schubert \cite{Schubert} proved that bridge number detects the unknot and that $\b - 1$ is additive under connected sum. 

Observe that if $h$ is a bridge position and if $t$ is a regular value of $h|_L$ separating the heights of the minima from the heights of the maxima, then $H = h^{-1}(t)$ is a sphere transverse to $L$ such that $|H \cap L|/2 = \b(h)$. The sphere $H$ separates $S^3$ into two 3-balls, each containing $\b(h)$ arcs. The link $L$ intersects each 3-ball in arcs that are isotopic, relative to their endpoints, into $H$; that is each 3-ball along with the arcs it contains is a trivial tangle. Any sphere separating $(S^3, L)$ into two trivial tangles is called a \defn{bridge sphere}; the arcs on either side are \defn{bridge arcs}. Given a bridge sphere $H$, it is possible to construct a bridge position $h$ such that $\b(h) = |H \cap L|/2$. In summary, the bridge number of a link can be calculated either by minimizing across certain height functions or by minimizing across bridge spheres. Each bridge sphere $H$ will correspond to a width tree having a single vertex. The label on the vertex will be $|H\cap L|/2  - 1$. In general, for any compact surface $S$ transverse to a 1--manifold $L$, we define the \defn{extent} of $S$ to be $\extent(S) = (-\chi(S) + |S \cap L|)/2$. (The Euler characteristic of $S$ is $\chi(S)$.)

Here is another way to calculate bridge number. Let $h\co S^3 \to I$ be \emph{any} height function for $L$. Let $t_0 < t_1 < \cdots < t_n$ be regular values of $h|_L$ such that:
\begin{itemize}
\item There is at least one critical value of $h|_L$ below $t_0$ and all such critical values correspond only to minima;
\item There is at least one critical value of $h|_L$ above $t_n$ and all such critical values correspond to maxima;
\item For each $i \in \{0, \hdots, n-1\}$, there is at least one critical value of $h|_L$ in $(t_i, t_{i+1})$ and either all critical values in the interval correspond to minima or all critical values in the interval correspond to maxima.
\end{itemize}
For each $i$, $h^{-1}(t_i)$ is a sphere transverse to $L$. These spheres come equipped with a transverse orientation arising from the standard orientation of $I$. The tangles consisting of $h^{-1}([-1,t_0])$ and $h^{-1}([t_n,1])$ and their intersections with $L$ are trivial tangles. For $i < n$, the submanifolds $h^{-1}([t_i, t_{i+1}])$ are homeomorphic to $S^2 \times I$. According to the definitions we give subsequently, their intersections with $L$ are also trivial tangles. If $i$ is even, then the sphere $h^{-1}(t_i)$ is a \defn{thick sphere}; if $i$ is odd, it is a \defn{thin sphere}. Let $\mc{H}^+$ be the union of the thick spheres and $\mc{H}^-$ the union of the thin spheres. Let $\mc{H} = \mc{H}^+ \cup \mc{H}^-$. We must have $\mc{H}^+ \neq \nil$. The height function is a bridge position if and only if $\mc{H}^-$ is empty. Thick spheres have bridge arcs and vertical arcs  (i.e. arcs isotopic to $I$-fibers in $S^2 \times I$)  incident to them, while thin spheres have only vertical arcs incident to them. Once we define ``width tree,'' it will be evident that the spheres $\mc{H}$ also give rise to a width tree; in this case a directed path. Each thick sphere is a vertex and there is a directed edge between two thick spheres if their heights are adjacent to the height of a thin sphere separating them. The direction on the edge is induced by the transverse orientation of the thin sphere. The labels on the vertices and edges are the extents of the corresponding spheres.

The surface $\mc{H}$ is an example of a \emph{multiple bridge surface} for $(S^3, L)$; the definition is due to Hayashi-Shimokawa \cite{HS} and is given subsequently. Indeed, adapting \cite{TT1}, we generalize it further to \defn{multiple c-bridge surfaces}. In the example at hand, all of the thick and thin spheres are concentric. In general, the thick and thin spheres of a multiple c-bridge surface need not be. The \defn{net extent} of $\mc{H}$ (equivalently, of the associated width tree) is defined to be
\[
\netextent(\mc{H}) = \extent(\mc{H}^+) - \extent(\mc{H}^-) = \sum\limits_{H \cpt \mc{H}^+} \extent(H) - \sum\limits_{F \cpt \mc{H}^-} \extent(F)
\]
The notation $\cpt$ means ``is a connected component of''; so the first sum is over all thick spheres and the second over all thin spheres. A pleasant exercise is to show that $\b(L) - 1$ is also equal to the minimum of $\netextent(\mc{H})$, where the minimum is taken over \emph{all} height functions $h$ for $L$, not just bridge positions. Theorem \ref{amalg} generalizes this further to show that $\b(L) - 1$ can be calculated by minimizing the net extent over all width trees associated to $(S^3, L)$.

From our height function $h$, we can define another classical invariant: Gabai width \cite{G3}. Our definition is due to Scharlemann-Schultens \cite{SS}. Given a height function $h \co S^3 \to I$ for $L$, the \defn{Gabai width} of $h$ is
\[
2\big(\sum\limits_{H \cpt \mc{H}^+} (|H \cap L|/2)^2 - \sum\limits_{F \cpt \mc{H}^-} (|F \cap L|/2)^2\big).
\]
The \defn{Gabai width} $\width_g(L)$ of the link is the minimum of this width over all height functions $h \co S^3 \to I$ for $L$. Gabai showed that if $\mc{H}$ achieves the minimum, then the thick spheres (i.e. components of $\mc{H}^+$) have topological significance for $L$. Thompson \cite{Thompson} showed that at least one thin sphere does as well. As with bridge number, Gabai width detects the unknot. For some time, it was an open question as to whether or not Gabai width is additive under connect sum; Blair and Tomova \cite{BT} showed it is not. It was also an open question as to whether or not every thin sphere has topological significance for $L$; Blair and Zupan \cite{BZ} showed they need not have topological significance.

Recently, Taylor and Tomova \cite{TT1, TT2} made small changes to the definition of Gabai width. These changes include normalizing the puncture count by the Euler characteristic of the sphere and relaxing the condition that the thick and thin spheres need to be concentric. After those changes, both thick and thin surfaces have topological significance, width still detects the unknot, and, rather surprisingly, width is now additive under connected sum. Their definition also applies to spatial graphs in (nearly) arbitrary 3-manifolds and to thick and thin surfaces of nonzero genus. The purpose of this paper is to provide a combinatorial reinterpretation of the simplest version of their invariants. This perspective sheds additional insight on the nature of those invariants and their relationship to the classical invariants. It may also ultimately help explain the nonadditivity of Gabai width as being related to the tree structure of the width trees of factors of a composite knot.

\subsection*{Acknowledgments} We are exceptionally grateful to Charles Parham for valuable conversation, insightful observations, and penetrating questions regarding this project. Thanks also to the referee for helpful suggestions. This work was partially funded by research grants from Colby College. 

\newpage

\section{Definitions, ditrees, and multiple c-bridge surfaces}

\subsection{Width trees and their invariants}
Figure \ref{fig:widthtree} gives an example of a width tree. Figure \ref{fig:multicbridge} shows a schematic diagram for a tangle decomposition associated to the width tree of Figure \ref{fig:widthtree}.

\begin{figure}[ht!]
\centering
\labellist
\small\hair 2pt
\pinlabel {3} [b] at 152 458
\pinlabel {0} [br] at 179 461
\pinlabel {4} [b] at 224 479
\pinlabel {0} [bl] at 254 455
\pinlabel {4} [bl] at 349 481
\pinlabel {3} [tl] at 320 452
\pinlabel {6} [l] at 290 422
\pinlabel {3} [br] at 256 397
\pinlabel {2} [bl] at 311 397
\pinlabel {4} [r] at 221 365
\pinlabel {3} [l] at 344 365
\endlabellist
\includegraphics[scale=1.0]{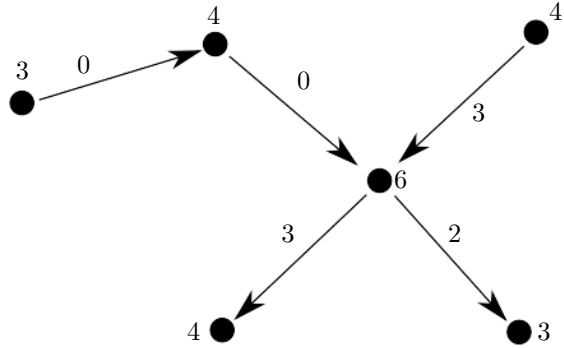}
\caption{An example of a width tree with no boundary vertices.}
\label{fig:widthtree}
\end{figure}

\begin{definition}\label{digraph term}
A directed graph whose underlying undirected graph is a tree is a \defn{ditree}.  A \defn{path} is a ditree that, ignoring directions on edges, is either a single vertex or has exactly two valence one vertices. A \defn{coherent path} is a path such that every valence two vertex in the path is the head of one edge and the tail of another. For a given ditree $T$ choose a subset of degree one vertices to be considered as \defn{boundary vertices}. (The subset may be empty, and not all leaves of $T$ need be boundary vertices.) Suppose that $\lambda$ is a function from the vertices and edges of $T$ to $\{x \in \Z : x \geq -1\}$. The pair $(T, \lambda)$ is a \defn{width tree} if the following hold:
\begin{enumerate}
\item If $v$ is a boundary vertex with incident edge $e$, then  $\lambda(v) = \lambda(e)$, and
\item For each vertex $v$, $\lambda(v)$ is at least the sum of the \emph{nonnegative} labels on the incoming edges to $v$ and is also at least the sum of the \emph{nonnegative} labels on the outgoing edges from $v$. 
\end{enumerate}

We call $\lambda$ a \defn{label function}. It is \defn{nonnegative} (resp. \defn{positive}) if all values of $\lambda$ are nonnegative (resp. positive). Two width trees $(T, \lambda)$ and $(T', \lambda')$ are \defn{equivalent} if there is a label-preserving graph isomorphism from $T'$ to $T$ which either preserves or reverses all orientations on edges. If $\lambda$ and $\lambda'$ are label functions, writing $\lambda \geq \lambda'$ means that for every vertex and edge $x$, $\lambda(x) \geq \lambda'(x)$. In particular a width tree $(T, \lambda)$ is positive if $\lambda \geq 1$.
\end{definition}

Inspired by the knot invariants ``net extent'' and ``width'' in \cite{TT2} and ``trunk'' in \cite{Ozawa}, we define the following invariants of width trees. As in \cite{dMPSS}, there are many other possibilities, but we focus on these.

\begin{definition}\label{inv. def.}
Suppose that $(T, \lambda)$ is a width tree with edge set $E$ and vertex set $V$. The \defn{net extent} is
\[
\netextent(T, \lambda) = \sum_{v \in V} \lambda(v) - \sum_{e \in E}\lambda(e).
\]
The \defn{width} is
\[
\width(T, \lambda) = 2\big(\sum_{v \in V} \lambda(v)^2 - \sum_{e \in E} \lambda(e)^2\big) = 2\netextent(T, \lambda^2) = \netextent(T, 2\lambda^2).
\]

The \defn{trunk} of a width tree $(T, \lambda)$ is the maximum of $\lambda(v)$ over all vertices $v$ of $T$. 
\end{definition}

Despite the connection between these invariants of width trees and link invariants, width trees in general are not particularly useful in illuminating the topology of a link. To rectify this, we introduce some additional structure. Its usefulness will become clear in subsequent sections.

\begin{definition}\label{def: slim}
Suppose that $(T, \lambda)$ is a width tree.  An edge $e$ of $T$ with an endpoint at a nonboundary vertex $v$ is a \defn{product edge} at $v$ if it is either the sole incoming edge to $v$ or the sole outgoing edge to $v$ and if $\lambda(e) = \lambda(v)$. The width tree $(T, \lambda)$ is \defn{productless} if it has no product edges. A \defn{distance threshold} for $(T, \lambda)$ is a fixed nonnegative integer $\delta = \delta(T, \lambda)$. A nonnegative, productless width tree with distance threshold $\delta \geq 2$ is called a \defn{slim width tree}.
\end{definition}

\subsection{Nearly standard definitions}
The 3-ball is denoted $B^3$; the $n$-sphere $S^n$. We let $I = [-1,1]$. All surfaces in this paper are orientable. A \defn{tangle} $(M, \tau)$ (for our purposes) consists of a properly embedded 1--manifold $\tau$ in a 3-manifold $M$ that is the result of removing a finite number (possibly zero) of open 3-balls from $S^3$. If $S \subset M$ is a properly embedded surface transverse to $\tau$, we write $S \subset (M, \tau)$. The \defn{punctures} of $S$ are the points $S \cap \tau$. We let $(M, \tau) \setminus S$ denote the result of cutting both $M$ and $\tau$ open along $S$. Similarly, we let $M \setminus \tau$ denote the result of removing an open regular neighborhood of $\tau$ from $M$. If $M'$ is a component of $M\setminus S$ and if $\tau' = \tau \cap M'$, we write $(M', \tau') \cpt (M,\tau)\setminus S$. 

Suppose $S, S' \subset (M, \tau)$ are either surfaces in a tangle or the boundary of the tangle.  We say that they are isotopic or \defn{$\tau$-parallel} if there is an isotopy in $M$ of $S$ to $S'$ such that at all times during the isotopy the surface is transverse to $\tau$. A simple closed curve in $S$ is \defn{inessential} if it is either isotopic in $S \setminus \tau$ to a component of $\boundary S$ (i.e. is \defn{$\boundary$-parallel}) or if it bounds an unpunctured or once-punctured disc in $S$; it is \defn{essential} otherwise. An unpunctured or once-punctured disc properly embedded in $(M, \tau)\setminus S$ with boundary in $S$ is an \defn{sc-disc} if it is not $\tau$-parallel to a disc in $S$, relative to its boundary. An sc-disc $D$ is a \defn{c-disc} if $\boundary D$ is essential in $S$. An unpunctured c-disc is a \defn{compressing disc}; a once-punctured c-disc is a \defn{cut disc}. If $D$ is an sc-disc that is not a c-disc, then $D$ is a \defn{semi-compressing disc} if it is unpunctured and a \defn{semi-cut disc} if it is once-punctured. If there is no compressing disc, then $S$ is \defn{incompressible}; if no c-disc, then $S$ is \defn{c-incompressible}.

A surface $S \subset (M, \tau)$ is \defn{$\boundary$-parallel} if it is parallel in $M \setminus \tau$ to $\boundary (M \setminus \tau)$. When $M \neq S^3$, being $\boundary$-parallel is a more general condition than being $\tau$-parallel to a component of $\boundary M$. A surface $S \subset (M, \tau)$ is \defn{essential} (resp. \defn{c-essential}) if it is incompressible (resp. \defn{c-incompressible}), not $\boundary$-parallel, not an unpunctured 2-sphere bounding a 3-ball in $M$ disjoint from $\tau$, and not a twice-punctured 2-sphere bounding a ball in $M$ whose intersection with $\tau$ is a single arc isotopic, relative to its endpoints, into $S$.  

A tangle $(M, \tau)$ is an \defn{elementary ball tangle} if $M$ is homeomorphic to $B^3$ and if $\tau$ is either empty or is a single arc isotopic relative to its endpoints into $\boundary M$. It is a \defn{product tangle} if $(M,\tau)$ is homeomorphic to $(S^2 \times I, \text{(points)} \times I)$. The tangle $(M,\tau)$ is \defn{even} if every component of $\boundary M$ has an even number of punctures.  The tangle $(M,\tau)$ is \defn{split} if it contains an essential unpunctured sphere. The tangle $(M,\tau)$ is \defn{irreducible} if it is not split and does not contain a once-punctured 2-sphere. Notice that an even, nonsplit tangle such that every component of $\boundary M$ has at least 2 punctures is necessarily irreducible. The tangle  $(M,\tau)$ is \defn{composite} if it is not an elementary ball tangle and if $(M,\tau)$ contains an essential twice-punctured 2-sphere. The tangle $(M, \tau)$ is \defn{prime} if it is neither composite nor an elementary ball tangle. A prime tangle is necessarily nonsplit, since otherwise, we could tube an inessential twice-punctured sphere to an essential unpunctured sphere to create an essential twice-punctured sphere. 

\subsection{Multiple c-bridge surfaces} \label{mult cbridge}
We turn now to the topological objects associated to width trees. There is a long intellectual history giving rise to these constructions. Gabai \cite{G3} gave the first definition of ``thin position'' for knots in $S^3$ using height functions as we described previously. Scharlemann-Thompson \cite{ST} reinterpreted the idea for 3-manifolds. A number of authors, then attempted, with varying degrees of success, to merge the definitions. Most notably, Heath-Kobayashi \cite{HK} showed that from a height function minimizing Gabai width it is possible to construct a tangle decomposition of $L$ by essential spheres that are closely related to the thick and thin spheres, but are not necessarily concentric. Hayashi-Shimokawa \cite{HS} developed this, constructing tangle decompositions by not necessarily concentric thick and thin spheres with useful properties. From Hayashi and Shimokawa's tangle decomposition it is possible to construct an associated height function on $(S^3, L)$ not to $\R$ but to a tree.  Tomova \cite{Tomova2} began the extension of Hayashi and Shimokawa's work to make use of cut-discs. Recently, Taylor-Tomova \cite{TT1} gave a thorough treatment applicable in a wide variety of contexts. In \cite{TT2}, they  adapted Gabai width to these contexts and showed that it becomes additive under connected sum. Indeed, they defined two families of link invariants for a link $L$: $\netextent_x(L)$ and $\width_x(L)$ where $x \geq -2$ is an even integer that parameterizes the families and is a certain measure of Euler characteristic. When $x = -2$, these invariants are naturally associated to tangle decompositions of $L$ (since $\chi(S^2) = 2$); for $x > -2$, they are naturally associated to decompositions of $(S^3, L)$ by surfaces of possibly higher genus. We do not consider those decompositions in this paper; though this paper can be considered as additional evidence that those higher genus invariants are appropriate analogues of bridge number. These definitions are due to Taylor-Tomova \cite{TT2}, relying heavily on \cite{Tomova2} and \cite{HS}. Our c-trivial tangles are examples of the vp-compressionbodies in \cite{TT1}. See Figure \ref{fig:c-trivial} for an example.

\begin{definition}
Suppose that $(C, \tau_C)$ is a tangle such that $\boundary C \neq \nil$ and with one component of $\boundary C$ designated as $\boundary_+ C$ (and called the \defn{positive boundary}) and the union of other components designated as $\boundary_- C$ (and called the \defn{negative boundary}).  The tangle $(C, \tau_C)$ is a \defn{c-trivial tangle} if there is a (possibly empty) collection $\Delta$ of pairwise disjoint sc-discs for $\boundary_+ C$ such that the result $(C, \tau_C) \setminus \Delta$ of $\boundary$-reducing $(C, \tau_C)$ using $\Delta$ is the disjoint union of product tangles and elementary ball tangles, each with their positive boundary inherited from that of $C$. If it is possible to choose $\Delta$ so that it contains only compressing discs and semicompressing discs, then $(C, \tau_C)$ is a \defn{trivial tangle}. Two c-trivial tangles are equivalent if there is a homeomorphism of pairs taking one to the other and preserving the positive boundaries.
\end{definition}

\begin{figure}[ht!]
\centering
\labellist
\small\hair 2pt
\pinlabel {$\boundary_+ C$} [bl] at 97 129
\endlabellist
\includegraphics[scale=1.0]{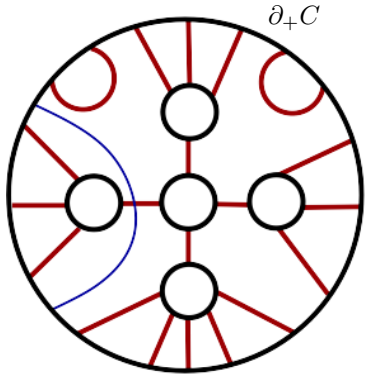}
\caption{An example of a c-trivial tangle $(C, \tau_C)$ with two bridge arcs, four ghost arcs, and fourteen vertical arcs. A cut disc for $\boundary_+ C$ is shown.}
\label{fig:c-trivial}
\end{figure}

If $(C, \tau_C)$ is a c-trivial tangle, then the components of $\tau_C$ can be categorized into three types. There are \defn{vertical arcs}, which are the components having one endpoint on $\boundary_+ C$ and one endpoint on $\boundary_- C$; the \defn{bridge arcs}, which are the components having both endpoints on $\boundary_+ C$; and the \defn{ghost arcs}, which are the components having both endpoints on $\boundary_- C$. We will have occasion to use the \defn{ghost arc graph} of a c-trivial tangle. This is the graph whose vertices are the components of $\boundary_- C$ and whose edges are the ghost arcs. In our setting, since $\boundary_+ C$ is a sphere, each component of the ghost arc graph is a tree. A c-trivial tangle is irreducible if and only if its negative boundary contains no unpunctured or once-punctured sphere. In that case, there are no semi-compressing discs that are not compressing discs. Similarly, a c-trivial tangle that is not an elementary ball tangle is prime and irreducible if and only if no negative boundary component is a sphere with two or fewer punctures. In this case, there are no semi c-discs that are not c-discs.

\begin{definition}
Suppose that $(C_i, \tau_i)$ is a c-trivial tangle for $i = 1,2$ and that $|\boundary_+ C_1 \cap \tau_1| = |\boundary_+ C_2 \cap \tau_2|$. Let $(M, \tau)$ be a tangle that results from gluing $(C_1, \tau_1)$ to $(C_2, \tau_2)$ via a homeomorphism of $\boundary_+ C_1$ to $\boundary_+ C_2$ taking punctures to punctures. Let $S$ be the image of $\boundary_+ C_1$ (equivalently $\boundary_+ C_2$) in $(M, \tau)$. We say that $S$ is a \defn{c-bridge surface} for $(M, \tau)$. If $S$ has a transverse orientation it is an \defn{oriented c-bridge surface}. If both $(C_1, \tau_1)$ and $(C_2, \tau_2)$ are trivial tangles then $S$ is an \defn{(oriented) bridge surface}. 

A \defn{multiple c-bridge surface} $\mc{H} = \mc{H}^+ \cup \mc{H}^-$ for a tangle $(M, \tau)$ is the union of pairwise disjoint spheres in $(M, \tau)$ such that the following hold:
\begin{enumerate}
\item Each component of $\mc{H}$ lies in precisely one of $\mc{H}^+$ or $\mc{H}^-$,
\item Each component of $(M, \tau) \setminus \mc{H}$ is a c-trivial tangle,
\item $\mc{H}^+$ is the union of the positive boundary components of the components of $(M, \tau) \setminus \mc{H}$ and $\boundary M \cup \mc{H}^-$ is the union of the negative boundary components.
\end{enumerate}
Furthermore, if $\mc{H}$ has a transverse orientation such that for each c-trivial tangle $(C, \tau_C)$ of $(M, \tau)\setminus \mc{H}$, if the transverse orientation on $\boundary_+ C$ points into $C$ then the transverse orientation on each component of $\boundary_- C$ points out of $C$ and vice versa, then $\mc{H}$ is an \defn{oriented multiple c-bridge surface}. The components of $\mc{H}^-$ are \defn{thin surfaces} and the components of $\mc{H}^+$ are \defn{thick surfaces}. For a given tangle $(M, \tau)$, we let $\H(M,\tau)$ denote the set of oriented multiple c-bridge surfaces up to isotopy transverse to $\tau$. See Figure \ref{fig:multicbridge} for an example of a multiple c-bridge surface. 
\end{definition}

\begin{remark}
Without the presence of ghost arcs, the resulting multiple c-bridge surface is the ``multiple Heegaard splitting'' of \cite{HS}. In our setting we will refer to this as a \defn{multiple bridge surface}. The corresponding dual trees were used in \cite{dMPSS} to relate computational properties of knot diagrams to the topological structure of knots. Our width trees are similar to the fork complexes of \cite{SSS}, but applied to multiple c-bridge surfaces rather than Heegaard splittings of 3-manifolds. 
\end{remark}

\begin{figure}[ht!]
\includegraphics[scale=0.5]{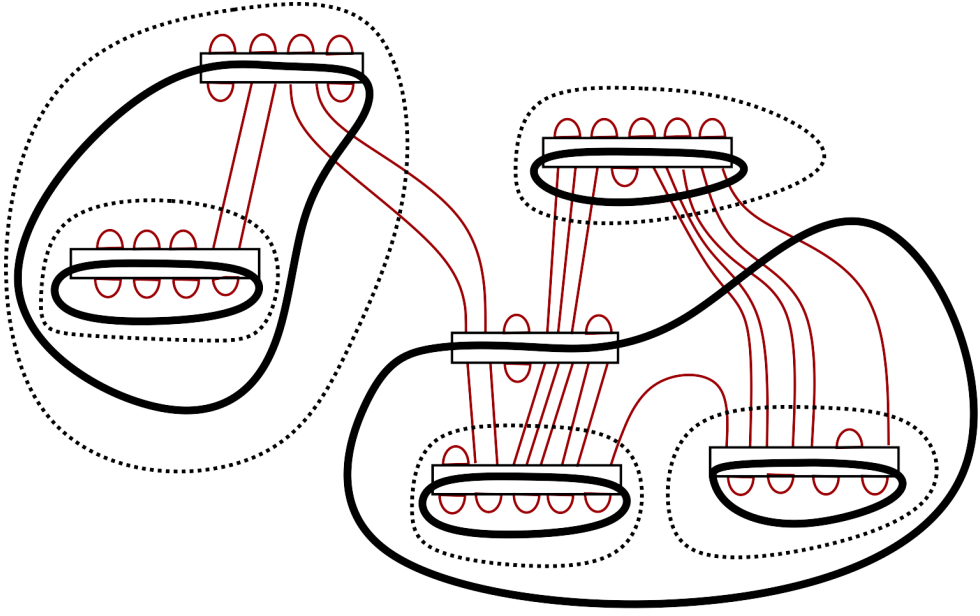}
\caption{An example of a multiple c-bridge surface  $\mc{H}^+$. The spheres of $\mc{H}^+$ are shown in thick lines and the spheres of $\mc{H}^-$ are shown in dashed lines. Braids should be placed into each of the boxes to obtain a knot or link $L$ such that $\mc{H} \in \H(M,L)$ after orienting $\mc{H}$. There is a single ghost arc in $L \setminus \mc{H}$.}
\label{fig:multicbridge}
\end{figure}

From an oriented multiple c-bridge surface $\mc{H}$ for a tangle or link $(M, \tau)$ we can construct an associated width tree. There is an orientation on the multiple c-bridge surface $\mc{H}$ in Figure \ref{fig:multicbridge} so that the width tree in Figure \ref{fig:widthtree} is associated to it.  When $M = S^3$, the width tree is essentially the dual tree to the surface $\mc{H}^-$, with edges oriented according to the orientations on $\mc{H}^-$. More precisely, we do the following: Take a vertex for each component of $\mc{H}^+ \cup \boundary M$, with the vertices corresponding to the components of $\boundary M$ being the boundary vertices of $T$. The label $\lambda$ on each vertex is defined to be the extent of the corresponding sphere. Suppose that  $(C_i, \tau_i) \cpt (M, \tau)\setminus \mc{H}$ for $i = 1,2$ are two c-trivial tangles incident along a component $F \cpt (\mc{H}^- \cap \boundary_- C_1 \cap \boundary_- C_2)$. Join the vertices $\boundary_+ C_1$ and $\boundary_+ C_2$ by an edge with orientation corresponding to the orientation on $F$. Finally, if $F \cpt \boundary M$ and $F \cpt \boundary_- C$ for a c-trivial tangle $(C_1, \tau_1)$, then join the vertex of $T$ corresponding to $F$ to the vertex corresponding to $\boundary_+ C$ by an edge with orientation induced by the orientation of $\boundary_+ C$.

\begin{definition}
A width tree is \defn{associated} to  $\mc{H} \in \H(M,\tau)$ if it is constructed in this way. We say that a width tree is \defn{associated} to a tangle $(M, \tau)$ if there exists $\mc{H} \in \H(M, \tau)$ such that the width tree is associated to $\mc{H}$. For a tangle or link $(M, \tau)$ we let $\T(M,\tau)$  be the set of all width trees associated to $(M, \tau)$. 
\end{definition}

We now turn to creating tangle invariants, beginning with invariants of multiple c-bridge surfaces.
\begin{definition}
Suppose that $\mc{H} \in \H(M, \tau)$ is associated with a width tree $(T, \lambda)$ with $(M,\tau)$ an even tangle. Define
\[\netextent(\mc{H}) = \netextent(T, \lambda) = \sum\limits_{H \cpt \mc{H}^+} \extent(H) - \sum\limits_{F \cpt \mc{H}^-} \extent(F) \] 
and
\[
\width(\mc{H}) = \width(T, \lambda) = \sum\limits_{H \cpt \mc{H}^+} \extent(H)^2 - \sum\limits_{F \cpt \mc{H}^-} \extent(F)^2.
\]
Also define
\[
\tr(\mc{H}) = \tr(T,\lambda) = \max \{ \extent(H) : H \cpt \mc{H}^+\}
\]
\end{definition}

\begin{definition}
Suppose that $(M,\tau)$ is an even tangle. The \defn{bridge number} of $(M,\tau)$ is the minimum of $\{1 + \extent(H)\}$ over all bridge spheres $H$ for $(M,\tau)$. Define $\netextent_{-2}(M,\tau)$, $\width_{-2}(M,\tau)$, and $\tr(M,\tau)$ to be the minima of $\netextent(\mc{H})$, $\width(\mc{H})$, and $\tr(\mc{H})$ over all $\mc{H}\in \H(M,\tau)$. Define the \defn{bridge width} $\width_b(M,\tau)$ to be the minimum of $\width(\mc{H})$ over all $\mc{H} \in \H(M,\tau)$ such that $\netextent(\mc{H}) = \netextent(M,\tau)$.
\end{definition}

\begin{remark}\label{classical rel}
The invariants $\netextent_{-2}(M,\tau)$ and $\width_{-2}(M,\tau)$ are due to Taylor--Tomova \cite{TT2}. They define the invariants for knots, links, spatial graphs, or 1--manifolds more generally, in most any 3--manifold. The subscripted $-2$ refers to the fact that we are considering only multiple c-bridge surfaces consisting of spheres. Much of what we do could be generalized beyond that context. If $M = S^3$ and $\tau$ is a link, then Gabai's width $\width_G(M,\tau)$ is obtained by taking the minimum over all $(T,\lambda) \in \T(M,\tau)$ such that $T$ is a coherent path of the quantity
\[
\width(T,\lambda) + 4\netextent(T,\lambda) - 2.
\]
Similarly, Ozawa's trunk \cite{Ozawa} and Blair--Zupan's bridge width \cite{BZ} can be defined as certain minima over $(T,\lambda) \in \T(M,\tau)$ such that $T$ is a coherent path (Definition \ref{digraph term}). Blair-Ozawa \cite{BO} establish connections between Gabai's thin position, especially Ozawa's trunk, and ``representativity'' of a knot. In particular, given an essential tangle decomposition of a knot, twice the representativity of the knot gives a lower bound on the minimum of the trunk of the two tangles. We expect similar results to hold in our setting, with nearly \emph{verbatim} proofs.
\end{remark}

For links, it turns out that the invariant $\netextent_{-2}$ coincides with one less than bridge number (thus, justifying the terminology ``bridge width'' above). The following is from the forthcoming paper \cite{TT3} of Taylor and Tomova. Since it has not yet appeared, we provide a sketch of the proof for those familiar with the amalgamation of Heegaard splittings \cite{Schultens}. (See also \cite[Theorem 4.1]{Saito} for a similar result using similar methods.)

\begin{theorem}\label{amalg}
Suppose that a link $L \subset S^3$ is associated to a width tree $(T, \lambda)$. Then there exists a bridge sphere $H$ for $L$ such that $|H \cap L|/2 - 1 = \netextent(T, \lambda)$. In particular, $\b(L) - 1$ is the minimum of $\netextent(T, \lambda)$ over any set of width trees $(T,\lambda)$ associated to $L$ that includes all associated width trees having a single vertex.
\end{theorem}

\begin{proof}
Suppose $(M, \tau)$ is a tangle with transversally oriented c-bridge surface $H$ dividing $(M, \tau)$ into c-trivial tangles $(C_1, \tau_1)$ below $H$ and $(C_2, \tau_2)$ above $H$. We may view $(M, \tau)$ as being constructed from $\boundary_- C_1$ by attaching 1-handles (possibly containing subarcs of $\tau$ as their cores)  to $\boundary_-C_1 \times \{1\} \subset \boundary_- C_1 \times I$ and then 2-handles (possibly containing subarcs of $\tau$ as cocores) and then 3-handles (each containing up to one subarc of $\tau$ that is $\boundary$-parallel) and product tangles. Suppose that $(M', \tau')$ is another such tangle and that $H'$ is a transversally oriented c-bridge surface for $(M', \tau')$ dividing it into $(C'_1, \tau'_1)$ and $(C'_2, \tau'_2)$ with $(C'_1, \tau'_1)$ below $H'$. Suppose also that $\boundary_- C'_1 = \nil$ and that $F = \boundary_- C'_2$ is connected and is also a component of $\boundary_- C_1$. Since $F = \boundary_- C'_2$ is connected, $\tau'_2$ does not contain any ghost arcs. We may reduce $(C'_2, \tau'_2)$ along a collection $\Delta$ of compressing discs (no cut-discs) to arrive at another copy $W$ of $F \times I$. The remnants of these compressing discs are unpunctured discs in $\boundary W$. Since those discs are unpunctured, we may extend the attaching regions of the 1-handles in $C_2$ that lie on $F \times \{1\} \subset C_2$ vertically through $C_1 \cup W$ to lie on $F\times \{1\} \subset C_1$ as unpunctured discs. It follows that we may construct a c-bridge surface $J$ for $(M \cup M', \tau \cup \tau')$, called the \defn{amalgamation} of $H$ and $H'$ along $F$, with $\extent(J) = \extent(H_1) + \extent(H_2) - \extent(F)$.

Now if $(T, \lambda)$ is associated to a multiple c-bridge surface $\mc{H}$ for $L$, we choose some edge $e$ of $T$ incident to a leaf of $T$. Let $F$ be the associated thin surface to $e$ and $H'$ and $H''$ the thick surfaces associated to the endpoints of $e$. We may amalgamate along $F$ to produce a new multiple c-bridge surface for $L$ having one less thick surface and one less thin surface. The effect on the associated width tree is to remove an edge and vertex (and change the labelling). Repeating this eventually produces a width tree with no edges but with the same netextent as $\mc{H}$. This width tree is associated to a c-bridge surface $H$ for $L$ with $\extent(H) = \netextent(\mc{H})$. Since every c-bridge sphere for a link in $S^3$ is also a bridge sphere, the result follows.
\end{proof}

The invariants we have introduced for width trees are related to each other. The proof of the next lemma involves a combinatorial version of amalgamation.

\begin{lemma}\label{relations}
Suppose that $(T, \lambda)$ is a non-negative width tree. Then
\[ 2\tr(T, \lambda)^2 \leq \width(T, \lambda) \leq 2\netextent(T, \lambda)^2. \]
\end{lemma}
\begin{proof}
Let $T_0 = T$ and $\lambda_0 = \lambda'_0 = \lambda$. If $T_i$ has an edge, construct $T_{i+1}$ by choosing a leaf $u$ of $T_i$ with incident edge $f$ and remove both $u$ and $f$ from $T_i$. Suppose that the other endpoint of $f$ is at a vertex $w$ of $T_i$. Let $\lambda_{i+1}$ and $\lambda'_{i+1}$ coincide with $\lambda_i$ and $\lambda'_i$ (respectively) on all vertices and edges of $T_{i+1}$ except $w$. Let $\lambda_{i+1}(w)$ be the maximum of $\lambda_i(w)$ and $\lambda_i(u)$. Let $\lambda'_{i+1}(w) = \lambda'_i(u) + \lambda'_i(w) - \lambda'_i(f)$. We thus arrive at a sequence of trees $T_0, \hdots, T_n$ with $T_n$ a single vertex $t$. For each $i$, $(T_i, \lambda_i)$ and $(T_i, \lambda'_i)$ are width trees. Since $T_n = t$, $\width(T_n, \lambda_n) = 2\lambda_n(t)^2$, $\width(T_n, \lambda'_n) = 2\lambda'_n(t)^2$, $\netextent(T_n, \lambda_n) = \lambda_n(t)$, and $\width(T_n, \lambda'_n) = 2\lambda'_n(t)^2$. Our desired inequalities hold by induction.
\end{proof}

Despite the connection with link invariants, width trees in full generality do not capture much of the topology of a link. The next section gives an extended example, and then we show how to put additional structure on the width trees to make them more useful.

\section{Interlude: Extended Examples}

Davies and Zupan \cite{DZ}, adapting earlier work of Blair-Tomova \cite{BT} and Scharlemann-Thompson \cite{ST-width}, introduced a family of knots, together with two height functions on those knots. Thinking of the height function as projection to a vertical axis, these height functions correspond to two particular configurations of the knots in $S^3$. The family depends on a choice of  4-tuple $(r_1, r_2, s_1, s_2)$ of natural numbers. The two configurations are denoted $k(r_1, r_2, s_1, s_2)$ and $k'(r_1, r_2, s_1, s_2)$ and are depicted on the left and right of Figure \ref{premoreex} respectively. We've drawn a more schematic version than that provided by Davies--Zupan.

\begin{figure}[ht!]
\centering
\labellist
\small\hair 1pt
\pinlabel {$A$} [bl] at 85 272
\pinlabel {$A$} [bl] at 261 173
\pinlabel {$B$} [bl] at 154 239
\pinlabel {$B$} [bl] at 294 271
\pinlabel {$C$} [tl] at 85 121
\pinlabel {$C$} [br] at 217 222
\pinlabel {$D$} [tl] at 153 148
\pinlabel {$D$} [tl] at 296 108
\pinlabel {$1$} [l] at 33 209
\pinlabel {$1$} [r] at 221 190
\pinlabel {$r_1$} [r] at 57 251
\pinlabel {$r_1$} [r] at 57 140
\pinlabel {$r_1$} [r] at 232 227
\pinlabel {$r_1$} [r] at 232 150
\pinlabel {$r_2$} [t] at 130 215
\pinlabel {$r_2$} [b] at 130 173
\pinlabel {$r_2$} [r] at 273 245
\pinlabel {$r_2$} [r] at 273 131
\pinlabel {$s_1$} [r] at 92 220
\pinlabel {$s_1$} [r] at 88 171
\pinlabel {$s_1$} [l] at 257 230
\pinlabel {$s_1$} [l] at 257 147
\pinlabel {$s_2$} [l] at 147 196
\pinlabel {$s_2$} [l] at 291 196
\endlabellist
\includegraphics[scale=1.0]{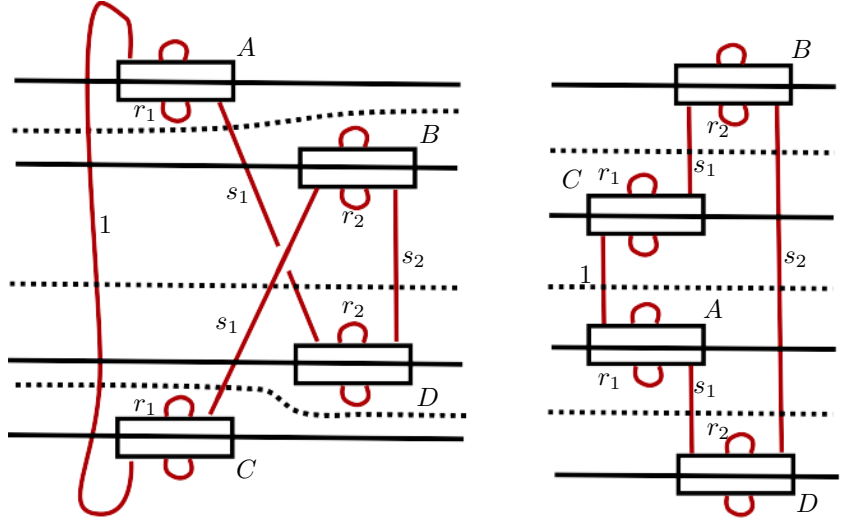}
\caption{The configurations $k(r_1, r_2, s_1, s_2)$ (on the left) and $k'(r_1, r_2, s_1, s_2)$ (on the right) of Davies-Zupan. For consistent choices of parameters and braids, $k$ and $k'$ are isotopic in $S^3$. Each box represents a braid of a certain number of strands. The cups and caps labelled with a natural number $n$ represent $n$ bridge arcs capping off strands of the braid on top or bottom, while the line segments labelled $n$ represent that many vertical arcs attached to the top or bottom of the braid. Unlabelled bridge arcs or vertical arcs represent however many bridge arcs or vertical arcs are needed to ensure that the top and bottom of each braid strand is attached to a vertical arc or bridge arc. For example, the braid in box $A$ has $2r_1 + s_1$ strands. The thick horizontal lines represent the thick spheres of the height function and the dashed lines represent the thin spheres of the height function.}
\label{premoreex}
\end{figure}

The isotopic configurations $k(3,0,3,3)$ and $k'(3,0,3,3)$ play an important role in Blair and Tomova's proof that Gabai's width is not additive. Davies and Zupan show that the configuration $k(2,1,3,7)$ has greater trunk and bridge number than that of $k'(2,1,3,7)$ but smaller Gabai width. They also show that the configuration $k(4,1,3,3)$ has greater bridge number and Gabai width than that of $k'(4,1,3,3)$, but smaller trunk. Key to the work of Blair-Tomova is the result that $k'(3,0,3,3)$ is in Gabai thin position, as long as the braids are sufficiently complicated. Blair-Zupan \cite{BZ} further extend their techniques. In \cite{TT2}, Taylor and Tomova show that this position is not ``locally thin'' in their sense. 

For both $k(r_1, r_2, s_1, s_2)$ and $k'(r_1, r_2, s_1, s_2)$, the thick and thin spheres are all concentric; consequently, the associated width trees are paths.  In Figure \ref{moreex}, for both $k(r_1, r_2, s_1, s_2)$ and $k'(r_1, r_2, s_1, s_2)$ we exhibit other multiple vp-bridge surfaces and their corresponding width trees. The widths (whether in the sense of Gabai or in the sense of this paper) for these new decompositions are strictly smaller than for those in Figure \ref{premoreex}. More significantly, the decomposition exhibited on the left of Figure \ref{moreex} is not locally thin. Theorem \ref{biggest one} below shows that if the braids are complicated enough, then the width tree on the right of Figure \ref{moreex} is the unique slim width tree for the knot.

\begin{figure}[ht!]
\centering
\labellist
\small\hair 1pt
\pinlabel {$r_1 + \frac{s_1}{2} - 1$} [b] at 45 134
\pinlabel {$r_1 + \frac{s_1}{2} - 1$} [t] at 45 9
\pinlabel {$r_2 + (s_1 + s_2)/2 - 1$} [b] at 153 134
\pinlabel {$r_2 + (s_1 + s_2)/2 - 1$} [t] at 153 11
\pinlabel {$\frac{s_1}{2} - 1$} [tr] at 65 96
\pinlabel {$\frac{s_1}{2} - 1$} [br] at 65 40
\pinlabel {$\frac{s_1 + s_2}{2} - 1$} [tl] at 134 96
\pinlabel {$\frac{s_1 + s_2}{2} - 1$} [bl] at 134 40
\pinlabel {$s_1 + (s_2 + 1)/2 - 1$} [l] at 107 74
\pinlabel {$r_2 + (s_1 + s_2)/2 - 1$} [r] at 358 151
\pinlabel {$r_1 + \frac{s_1}{2} - 1$} [r] at 282 121
\pinlabel {$r_1 + \frac{s_1}{2} - 1$} [r] at 283 48
\pinlabel {$r_2 + (s_1 + s_2)/2 - 1$} [r] at 358 6
\endlabellist
\includegraphics[scale=1.0]{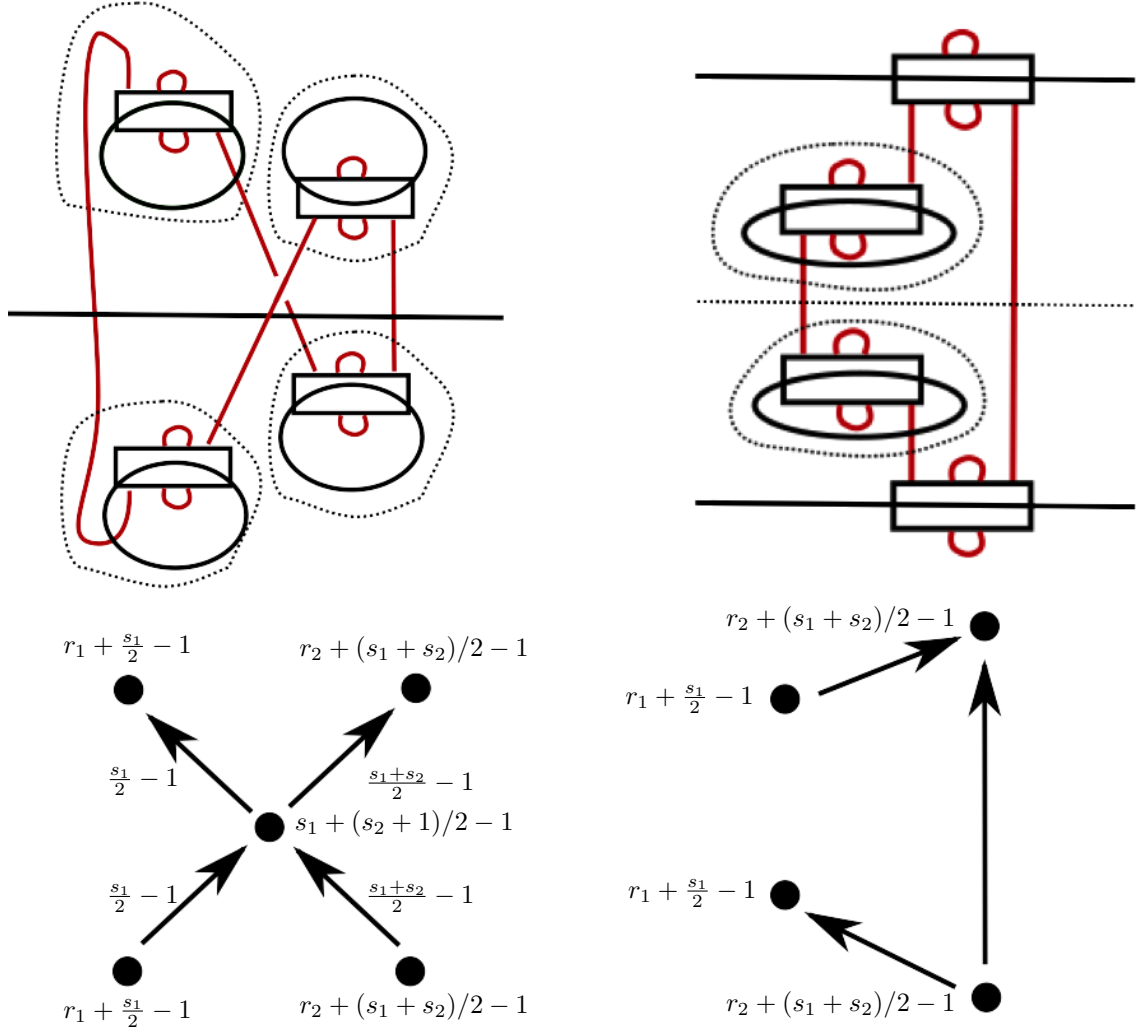}
\caption{The configurations $k(r_1, r_2, s_1, s_2)$ (on the left) and $k'(r_1, r_2, s_1, s_2)$ (on the right) with new multiple vp-bridge surfaces, along with the associated width trees. We have omitted the labels that are the same as in Figure \ref{premoreex}.}
\label{moreex}
\end{figure}

\section{Slim Width Trees}\label{Slim Width Trees}

Recall that in Definition \ref{def: slim}, we defined the notion of a \emph{slim width tree}. These are those width trees that are productless and satisfy a distance threshold of at least 2. Let $\T_2(M,\tau) \subset \T(M,\tau)$ be the set of slim width trees. In this section, we reinterpret recent work of Taylor--Tomova to show the following theorem. The first part of the theorem is proved as Proposition \ref{slim exist} and the second as Proposition \ref{props}.
\begin{theorem}\label{thm: slim main}
Suppose that $(M,\tau)$ is an even tangle such that $\tau \neq \nil $, each component of $\boundary M$ is c-incompressible and has at least 4 punctures, and  $(M,\tau)$ is not a product tangle. Then $\T_2(M,\tau)$ is nonempty and the following hold:
\begin{enumerate}
\item There exists $(T, \lambda) \in \T_2(M,\tau)$ such that $\netextent(T,\lambda) = \netextent_{-2}(M,\tau)$ and $\width(T,\lambda) = \width_b(M,\tau)$.
\item There exists $(T, \lambda) \in \T_2(M,\tau)$ such that $\width(T,\lambda) = \width_{-2}(M,\tau)$. 
\item For each $(T, \lambda) \in \T_2(M,\tau)$, there is an edge $e$ of $T$ such that $\lambda(e) = -2$ if and only if $(M,\tau)$ contains an essential unpunctured sphere.
\item For each $(T, \lambda) \in \T_2(M,\tau)$, there is an edge $e$ of $T$ such that $\lambda(e) = 0$ if and only if $(M,\tau)$ is composite.
\item For each $(T, \lambda) \in \T_2(M,\tau)$, if there is an edge $e$ of $T$ such that $\lambda(e) = 1$, then $(M,\tau)$ contains an essential Conway sphere or a component of $\boundary M$ is a 4-punctured sphere.
\end{enumerate}
\end{theorem}

\begin{remark}
Although the set $\T_2(M,\tau)$ is most always defined, it is not always particularly interesting. For instance, if a link $L$ is mp-small (i.e. has no essential meridional planar surfaces), then any associated slim width tree $(T, \lambda)$ consists of a single vertex $v$ with $\lambda(v) = \b(L) - 1$. However, for links that are not mp-small, we believe that slim width trees will prove to be a helpful way of encapsulating certain topological structure. For example, the failure of Gabai width to be additive under connected sum \cite{BT} or to have all thin spheres essential \cite{BZ} may be due to the fact that the associated slim width trees need not be paths. As indicated by the previous section, the examples of Blair-Tomova \cite{BT} (see also \cite[Section 6]{TT2}), Blair-Zupan \cite{BZ}, and Davies-Zupan \cite{DZ} are illuminating this regard. These sets of examples depend on the construction of a certain knot such that there are associated slim width trees $(T, \lambda)$ that are not coherent paths. By way of comparison, in Taylor-Tomova's theory \cite{TT1}, all the thin surfaces associated to edges of a slim width tree are essential.
\end{remark}

The following definitions will be used in subsequent sections as well. They are inspired by work of Hempel \cite{Hempel} on Heegaard splittings of 3-manifolds and many subsequent papers.  By \cite{Harvey} this distance is well-defined.

\begin{definition}\label{curve distance}
Suppose that $H$ is a sphere with $p \geq 4$ punctures and that $\gamma, \gamma'$ are two (not necessarily disjoint) essential simple closed curves on $H$. We say that $d(\gamma, \gamma') =  n$ if $n \geq 0$ is the smallest number such that there exists a sequence of essential simple closed curves
\[
\gamma_0, \gamma_1, \hdots, \gamma_n
\]
with $\gamma_0$ isotopic to $\gamma$, $\gamma_n$ isotopic to $\gamma'$ and such that $\gamma_i$ and $\gamma_{i+1}$ (for $i \in \{0, \hdots, n-1\}$) are disjoint if $p > 4$ and intersect minimally exactly twice if $p = 4$. If $\gamma, \gamma'$ are inessential curves in $H$, we say that $d(\gamma, \gamma') = 0$ if they are isotopic and $d(\gamma, \gamma') = 1$ otherwise (i.e. they are not isotopic but can be isotoped to be disjoint).
\end{definition}

\begin{definition}
Suppose that $(M,\tau)$ is an even tangle such that each component of $\boundary M$ has at least four punctures. Let $\mc{H} \in \H(M,\tau)$ and  $H \cpt \mc{H}^+$. If $|H \cap \tau| \geq 4$,  we define the \defn{sc-disc distance} $d_{sc}(H)$ of $H$ to be the minimal $n \geq 0$ such that there are sc-discs $A$ and $B$ for $H$ in $M \setminus \mc{H}$ on opposite sides of $H$ such that $d(\boundary A, \boundary B) = n$. If there are no such discs or if $|H \cap \tau| \leq 2$, we declare $d_{sc}(H) = \infty$.
\end{definition}

\begin{definition}\label{distance assoc}
If $(T, \lambda, \delta)$ is a width tree with distance threshold and if $(T, \lambda)$ is associated to $\mc{H} \in \H(M, \tau)$, we say that $(T, \lambda, \delta)$ is \defn{associated to} $\mc{H}$ if $d_{sc}(H) \geq \delta$ for every component $H \cpt \mc{H}^+$.
\end{definition}

Taylor and Tomova \cite{TT1} defined a partial order denoted $\to$ and called \defn{thins to} on $\H(M, \tau)$. This partial order is a more complicated version of the ``destabilization'' and ``weak reduction'' operations from Heegaard splitting theory and the thinning operation of Hayashi-Shimokawa. A minimal element under this partial order is said to be \defn{locally thin}. The first statement of the next theorem is a more specific version of \cite[Theorem 6.17]{TT1}. The second statement concerning $\netextent$ and $\width$ is an application of \cite[Corollary 3.3]{TT2}.

\begin{theorem}[Taylor-Tomova]\label{thm: poset}
Let $(M, \tau)$ be a tangle containing no once-punctured spheres. Assume that no component of $\boundary M$ is a sphere intersecting $\tau$ two or fewer times. Then for every $\mc{J} \in \H(M, \tau)$, there exists a locally thin $\mc{H} \in \H(M, \tau)$ such that $\mc{J} \to \mc{H}$. The invariants $\netextent$ and $\width$ are non-increasing under $\to$.
\end{theorem}

\begin{definition}
Suppose that $\mc{H} \in \H(M, \tau)$. We say that $\mc{H}$ is \defn{perturbed} if there exists $H \cpt \mc{H}^+$ such that the following holds. There exist bridge arcs $\tau_1, \tau_2 \cpt \tau \setminus \mc{H}$ with endpoints on $H$ such that there are disks $D_1$ and $D_2$ with interior disjoint from $\mc{H} \cup \tau$ and with $\boundary D_i$ the union of $\tau_i$ and an arc on $H$ and those arcs share an endpoint and are otherwise disjoint. $\mc{H}$ is \defn{sc-strongly irreducible} if no two sc-discs for $H \cpt \mc{H}^+$ in $M \setminus \mc{H}^-$ and on opposite sides of $H$ have disjoint boundaries.
\end{definition}

By \cite{TT1} a locally thin multiple c-bridge surface is always sc-strongly irreducible and unperturbed.

\begin{proposition}\label{slim exist}
Suppose that $(M,\tau)$ is an even tangle such that $\tau \neq \nil $, $(M,\tau)$ is not a product tangle, and each component of $\boundary M$ is c-incompressible and has at least four punctures. Then $\T_2(M,\tau) \neq \nil$ and:
\begin{enumerate}
\item There exists $(T, \lambda) \in \T_2(M,\tau)$ such that $\netextent(T,\lambda) = \netextent_{-2}(M,\tau)$ and $\width(T,\lambda) = \width_b(M,\tau)$.
\item There exists $(T, \lambda) \in \T_2(M,\tau)$ such that $\width(T,\lambda) = \width_{-2}(M,\tau)$. 
\end{enumerate}
In particular, if $(T, \lambda)$ is associated to a locally thin $\mc{H} \in \H(M,\tau)$, then $(T,\lambda)$ is slim.
\end{proposition}

\begin{proof}
Suppose that $\mc{J} \in \H(M,\tau)$. Since each tangle has at least one bridge surface, such an $\mc{J}$ exists. Choose $\mc{J}$ so that either $\netextent(\mc{J}) = \netextent_{-2}(M,\tau)$ and $\width(\mc{J}) = \width_b(M,\tau)$ or so that $\width(\mc{J}) = \width(M,\tau)$. Since $(M,\tau)$ is even, it contains no once-punctured spheres. By Theorem \ref{thm: poset}, there exists a locally thin $\mc{H} \in \H(M,\tau)$ such that $\mc{J} \to \mc{H}$. Let $(T,\lambda)$ be the width tree associated to $\mc{H}$. Since $\netextent$ and $\width$ are nonincreasing under $\to$, we have either:
\begin{itemize}
\item $\netextent(T,\lambda) = \netextent_{-2}(M,\tau)$ and $\width(T,\lambda) = \width_b(M,\tau)$, or
\item $\width(T,\lambda) = \width_{-2}(M,\tau)$,
\end{itemize}
according to our original choice of $\mc{J}$. With slightly greater generality, assume that $\mc{H} \in \H(M,\tau)$ is any locally thin multiple c-bridge surface and let $(T, \lambda)$ be the associated width tree. We will show that $(T, \lambda)$ is slim.

By the properties of being locally thin \cite[Theorem 7.6]{TT1}:
\begin{itemize}
\item no component of $\mc{H}^-$ is $\boundary$-parallel in $M \setminus \tau$;
\item no component $(W, \tau \cap W)$ of $(M,\tau)\setminus \mc{H}$ is a product tangle with one component of $\boundary W$ in $\mc{H}^-$ and the other in $\mc{H}^+$.
\end{itemize}

Suppose that $(T, \lambda)$ has a product edge $e$ at vertex $v$. Let $H \cpt \mc{H}^+$ and $F \cpt \mc{H}^- \cup \boundary M$ be the surfaces associated with $v$ and $e$, respectively. If $e$ does not have an endpoint at a boundary vertex of $T$, then $F \cpt \mc{H}^-$ and if it does, then $F \cpt \boundary M$. Since $e$ is a product edge, $H$ and $F$ have the same number of punctures and cobound a component $(W, \tau \cap W)$ of $(M,\tau) \setminus \mc{H}$. Consequently, $(W, \tau \cap W)$ is a product tangle. Thus, $F \cpt \boundary M$. Let $(U, \tau \cap U) \cpt (M,\tau) \setminus \mc{H}$ be the component on the other side of $H$. If $(U, \tau \cap U)$ contains an sc-disc $D$ for $H$, we can extend $\boundary D$ through the product tangle $(W, \tau \cap W)$ and conclude that $F$ has a c-disc, contradicting our hypothesis that $\boundary M$ is c-essential. Thus, $(U, \tau \cap U)$ admits no sc-disc for $H$. Thus, it is either a product tangle or an elementary ball tangle. In the latter case, $|H \cap \tau| \leq 2$ and so $|F \cap \tau| \leq 2$, a contradiction. In the former case, we see that $(M,\tau)$ is itself a product tangle. This again contradicts our hypotheses. Thus, $(T,\lambda)$ is productless. 

We desire to show that $(T,\lambda, 2) \in \T(M,\tau)$.  Let $H \cpt \mc{H}^+$. Let $W \cpt (M,\tau) \setminus \mc{H}^-$ be the component containing $H$. Let $p = |H \cap \tau|$. Since $(M,\tau)$ is an even tangle, $p$ is even. Suppose that $p \geq 6$. If $d_{sc}(H) \leq 1$, then there exist disjoint sc-discs in $W$ for $H$ on opposite sides of $H$. This means that $H$ is sc-weakly reducible, contradicting the properties of local thinness \cite[Theorem 7.6]{TT1}. Thus, if $p \geq 6$, then $d_{sc}(H) \geq 2$. 

Suppose that $p = 4$ and $d_{sc}(H) \leq 1$. Since $H$ is sc-strongly irreducible, there exist c-discs $A$ and $B$ for $H$ on opposite sides of $H$ with $\boundary A$ and $\boundary B$ intersecting minimally exactly twice. The curves $\boundary A$ and $\boundary B$ each bound twice-punctured discs $D_A$ and $D_B$, respectively, in $H$. If $A$ or $B$ were a cut disc, $(M, \tau)$ would contain a sphere with 3 punctures. This contradicts the fact that $(M,\tau)$ is an even tangle and that every sphere in $S^3$ separates. Thus, $A$ and $B$ are both compressing discs. Let $\tau_A$ and $\tau_B$ be the components of $\tau \cap W$ with endpoints in $D_A$ and $D_B$, respectively.  

Either $\tau_A$ is the union of two vertical arcs or $\tau_A$ is a single bridge arc. In the former case, since $A \cup D_A$ is a twice-punctured sphere and since $(M,\tau)$ contains no once-punctured spheres, there is a twice-punctured sphere $S \cpt \mc{H}^- \cap \boundary W$. Take a parallel copy of this sphere and tube along one of the arcs of $\tau_A$ to create a semi-cut disc for $H$. Its boundary is disjoint from $B$, and so $H$ is sc-weakly reducible, a contradiction. Thus, $\tau_A$ is a single bridge arc. A symmetric argument shows that $\tau_B$ is also a single bridge arc.  Since $\boundary A$ and $\boundary B$ intersect in a single point, we can construct bridge discs for $\tau_A$ and $\tau_B$ that intersect only in an endpoint. Hence, $H$ is perturbed. By \cite[Definition 6.16]{TT1}, it is possible to apply a thinning move to $\mc{H}$, contrary to the hypothesis that $\mc{H}$ is locally thin. Thus, $(T, \lambda, 2) \in \T(M,\tau)$ and so $\T_2(M,\tau) \neq \nil$ and the result holds.
\end{proof}

We now consider the connection between certain labels on a slim width tree and surfaces in the tangle.

\begin{proposition}\label{props}
Suppose that $(M,\tau)$ is an even tangle such that $\tau \neq \nil $, and each component of $\boundary M$ is c-incompressible and has at least four punctures. Suppose that $(T,\lambda) \in \T_2(M,\tau)$. If $(T, \lambda)$ is associated to $\mc{H}$, then each component of $\mc{H}^-$ is c-incompressible in $(M,\tau)$; no component of $(M,\tau)\setminus \mc{H}$ is a product tangle; each component of $\mc{H}^+$ is sc-strongly irreducible and unperturbed; and all such spheres with at most two punctures are essential. In particular, the following hold:
\begin{enumerate}
\item For each $(T, \lambda) \in \T_2(M,\tau)$, there is an edge $e$ of $T$ such that $\lambda(e) = -1$ if and only if $(M,\tau)$ is split.
\item For each $(T, \lambda) \in \T_2(M,\tau)$, there is an $e$ of $T$ such that $\lambda(e) \leq 0$ if and only if $(M,\tau)$ is split or composite.
\item For each $(T, \lambda) \in \T_2(M,\tau)$, if there is an edge $e$ of $T$ such that $\lambda(e) = 1$, then $(M,\tau)$ contains an essential Conway sphere or a component of $\boundary M$ is a 4-punctured sphere.
\end{enumerate}
\end{proposition}
\begin{proof}
We explain the general statements concerning $\mc{H}$; the specific conclusions (1), (2), and (3) are addressed at the end of the proof. Suppose that $(T, \lambda) \in \T_2(M)$ and that it is associated to $\mc{H} \in \H(M,\tau)$. We show that $\mc{H}$ satisfies many of the properties from \cite[Theorem 7.6]{TT1} of being locally thin. Suppose that tangle of $(M,\tau) \setminus \mc{H}$ is a product tangle between $F \cpt \mc{H}^-$ and $H \cpt \mc{H}^+$. The edge of $T$ corresponding to $F$ is then the sole outgoing or the sole incoming edge at the vertex corresponding to $H$. Since the tangle is a product tangle, $F$ and $H$ have the same number of punctures, and so the edge is a product edge. This contradicts the fact that $(T, \lambda)$ is productless. Similarly, no component of $(M,\tau) \setminus \mc{H}$ is a product tangle between a component of $\boundary M$ and a component of $\mc{H}^+$. 

Suppose that $H \cpt \mc{H}^+$. Since $\mc{H}$ satisifies the distance threshold of 2, $d_{sc}(H) \geq 4$. Consequently, $H$ is sc-irreducible. If $H$ is perturbed, let $D_A$ and $D_B$ be bridge discs on opposite sides of $H$ in $M\setminus \mc{H}^-$ such that the arcs $\boundary D_A \cap H$ and $\boundary D_B \cap H$ intersect only in a single point which is an endpoint of both arcs. Since $(M,\tau)$ is an even tangle, this implies that $|H \cap \tau| \geq 4$. The frontiers $A$ and $B$ of $D_A$ and $D_B$ are then compressing discs for $H$ in $M \setminus \mc{H}^-$ on opposite sides of $H$ whose boundaries intersect exactly twice. If $|H \cap \tau| = 4$, this implies $d_{sc}(H) \leq 1$, a contradiction. Suppose that $|H \cap \tau| > 4$. In this case, let $A'$ be the result of performing a ``finger-move'' of $\boundary A$ along the arc $\boundary D_B \cap H$, so as to slide $\boundary A$ across the puncture that is the endpoint of $\boundary D_B \cap H$ not shared by $\boundary D_A \cap H$. Let $\gamma$ be the resulting curve in $H$. It is essential since it bounds a thrice-punctured disc in $H$ and $H$ has more than four punctures. The curve $\gamma$ obviously bounds a cut-disc for $H$ in $M \setminus \mc{H}^-$ on the same side of $H$ as $D_A$. It also bounds a cut-disc for $H$ in $M \setminus \mc{H}^-$ on the same side of $H$ as $B$. To see this, observe we can also obtain it (up to isotopy in $H \setminus \tau$) by sliding $\boundary B$ along the arc $\boundary A \cap H$. Consequently, $H$ is sc-weakly reducible, a contradiction. We conclude that each $H \cpt \mc{H}^+$ is sc-irreducible and unperturbed.

That each component of $\mc{H}^-$ is c-essential follows from \cite[Corollary 7.5]{TT1}. If $\mc{H}$ is locally thin, then by \cite[Theorem 7.5]{TT1}, each component of $\mc{H}^-$ is also c-essential. We do not quite have the full strength of local thinness here, but we can show that if $P \cpt \mc{H}^-$ has fewer than two punctures, then it is c-essential. 

Suppose, first that $P \subset \mc{H}$ is an unpunctured sphere bounding a ball $B \subset (M \setminus \tau)$. If the interior of $B$ contains a component of $\mc{H}$, then that component of $\mc{H}^-$ is another unpunctured sphere. Thus, if $\mc{H}^-$ contains a c-inessential unpunctured sphere, by passing to an innermost such component  of $\mc{H}^-$, we see that $(M,\tau)\setminus \mc{H}$ would contain an $(S^2 \times I, \nil)$ component, a contradiction. Thus, no component of $\mc{H}^-$ is an inessential unpunctured sphere. Let $P \cpt \mc{H}^-$ be twice-punctured and $\boundary$-parallel. Let $W \subset M \setminus \tau$ be the product submanifold that $P$ bounds with an annulus  subsurface of $\boundary (M \setminus \tau)$. As $\mc{H}^-$ is c-incompressible and contains no inessential unpunctured spheres, any other component of $\mc{H}^- \cap W$ is also a $\boundary$-parallel twice-punctured sphere, so by choosing $P$ to be innermost, we may assume that the interior of $W$ is disjoint from $\mc{H}^-$ and contains a unique component $H \cpt \mc{H}^+$. By our choice of $P$, there are no ghost arc components of $(\tau \cap W) \setminus H$. That is, $H$ is a bridge surface for $(W, \tau \cap W)$. After capping off $\boundary W$ with a 3-ball containing an unknotted arc, $H$ becomes a bridge sphere for the unknot. By \cite{HS2}, it is either perturbed or twice-punctured. If it is perturbed, it was perturbed prior to capping off $\boundary W$, a contradiction. If it is twice-punctured, then in $(M,\tau) \setminus \mc{H}$, the spheres $H$ and $P$ cobound a product tangle, another contradiction. Thus, each sphere of $\mc{H}^-$ with two or fewer punctures is essential. since $\boundary M$ contains no sphere with two or fewer punctures, if $T$ has an edge $e$ with $\lambda(e) = -1, 0$, then $(M,\tau)$ is split or composite, respectively.

Suppose, therefore, that $e$ is an edge of $(T, \lambda)$. If $e$ is incident to a boundary vertex $v$ of $T$, then $\lambda(v) = \lambda(e)$. By our hypotheses on $\boundary M$, this implies that $\lambda(e) \geq 1$. In particular, if $e$ is incident to a boundary vertex of $T$, then $\lambda(e) = 1$, implies a component of $\boundary M$ is a four-punctured sphere. Suppose that $e$ is not incident to a boundary vertex and let $F \cpt \mc{H}^-$ be the associated thin sphere. By construction, $\lambda(e) = -1 + |F\cap\tau|/2$. Thus, if $\lambda(e) = -1$, then $F$ is unpunctured and if $\lambda(e) = 0$, then $F$ is twice-punctured. In either case, by our previous remarks, $F$ is essential. If $\lambda(e) = 1$, then $F$ is a c-incompressible Conway sphere. If it is inessential, it is $\boundary$-parallel in $(M \setminus \tau)$. This implies that either $\boundary M$ contains a four-punctured sphere or that it contains two thrice-punctured spheres. The latter possibility contradicts our assumption that $(M,\tau)$ is even. Thus, if $(T,\lambda)$ has an edge with label 1, then it either contains an essential Conway sphere or $\boundary M$ contains a four-punctured sphere.
 
Suppose now that $S \subset (M,\tau)$ is an essential sphere with two or fewer punctures. If $\mc{H}$ is locally thin, then \cite[Theorem 8.2]{TT1} shows that there is $F \cpt \mc{H}^-$ such that $|F \cap \tau| \leq |S\cap \tau|$. The edge $e$ of $T$ corresponding to $F$ has $\lambda(e) \in \{-1,0\}$. In our setting $\mc{H}$ may not be locally thin, but the result still holds. To see this, observe that in the proofs of \cite[Theorem 8.2]{TT1} and \cite[Theorem 7.2]{TT1} to which \cite[Theorem 8.2]{TT1}  refers, all that is used is the fact that there is no once-punctured sphere in $(M,\tau)$, that each component of $\mc{H}^-$ is c-incompressible, and that each component of $\mc{H}^+$ is sc-strongly irreducible. Consequently, \cite[Theorem 8.2]{TT1} holds in somewhat more generality than is indicated in \cite{TT1} and our result holds.
\end{proof}

The next lemma will be useful in the final section of this paper.
\begin{lemma}\label{thin not parallel}
Suppose that $(M, \tau)$ is even, prime, not a product tangle, and that each component of $\boundary M$ is c-incompressible and has at least four punctures. Let $(T, \lambda) \in \T_2(M,\tau)$ be associated with $\mc{H}$. Then either no two components of $\mc{H}^- \cup \boundary M$ are $\tau$-parallel, or there exists $\mc{J} \in \H(M, \tau)$ with $\netextent(\mc{J}) < \netextent(\mc{H})$ and $\width(\mc{J}) < \width(\mc{H})$.
\end{lemma}
\begin{proof}
Suppose that $F_0$ and $F_1$ are two components of $\mc{H}^- \cup \boundary M$ that are $\tau$-parallel. At least one must belong to $\mc{H}^-$.  Let $W$ be the product tangle between them. By Proposition \ref{props}, each component of $\mc{H}^-$ is c-incompressible. Thus, by \cite{Waldhausen}, any component of $\mc{H}^-$ in the interior of $W$ is $\tau$-parallel to $F_0$ and $F_1$. Consequently, we may assume that the interior of $W$ is disjoint from $\mc{H}^-$ and contains a single component $H \cpt \mc{H}^+$. If there are no ghost arcs in $(\tau \cap W) \setminus H$, $H$ is a bridge sphere for $(W, \tau \cap W)$. In that case, by \cite{HS2}, either $H$ is perturbed or is $\tau$-parallel to $F_0$ and $F_1$ within $W$. The first possibility contradicts Proposition \ref{slim exist} and the second contradicts Proposition \ref{props}. 

 Thus, $(\tau \cap W) \setminus H$ must contain a ghost arc $\kappa$. Since $H, F_0, F_1$ are all spheres, by the definition of c-trivial tangle, there is exactly one such ghost arc and $H$ does not separate $F_0$ and $F_1$. 

Let $a$ be the number of arcs of $\tau \cap W$. Since $(M, \tau)$ is prime, $a \geq 3$. Let $H' \subset W$ be a surface $\tau$-parallel to $F_0$ and $F_1$. Note that $H'$ is a bridge sphere for $(W, \tau \cap W)$. Also,
\[
\extent(H') = a-1 < 2a - 3 \leq \extent(H).
\]
Let  $\mc{J}^+ = (\mc{H}^+ \setminus H) \cup H'$; $\mc{J}^- = \mc{H}^-$; and $\mc{J} = \mc{J}^+ \cup \mc{J}^-$. Observe that $\mc{J}\in \H(M, \tau)$ and that $\netextent(\mc{J}) < \netextent(\mc{H})$ and $\width(\mc{J}) < \width(\mc{H})$.
\end{proof}

\section{Ditree geometry}
The geometry of a directed graph is closely connected to the structure of its sources and sinks. In this section, we relate this structure to the net extent of a width tree. We adapt standard terminology and techniques from graph theory \cite[Chapter 4.3]{West}. Throughout this section, for convenience, we consider only width trees without boundary vertices.

\begin{definition}
Suppose that $T$ is a ditree. A \defn{source} of $T$ is a vertex with no incoming edges and a \defn{sink} of $T$ is a vertex with no outgoing edges. A \defn{strong source-sink cut} for $T$ is a subset $S$ of the vertices of $T$ that contains all source vertices, contains no sink vertices, and for which there are no edges with a tail not in $S$ and a head in $S$. The \defn{size} $|\boundary S|$ of such a cut is the number of edges with tail in $S$ and head not in $S$. Let $N^2_-(T)$ denote the number of vertices of $T$ with a single incoming edge and $N^2_+(T)$ the number with a single outgoing edge. 
\end{definition}

Our main result is an adaptation of the classical Max Flow-Min Cut theorem for network flows; though in our case it is a ``Min Flow-Max Cut'' theorem. The proof will be given at the end of this section; after preliminary material. Figure \ref{fig:widthtreeeq} shows a width tree achieving equality in our lower bound.

\begin{theorem}\label{minflowmaxcut}
Suppose that $(T, \lambda)$ is a productless positive width tree. Let $\max |\boundary S|$ be the maximum size of a strong source-sink cut on $T$. Then
\[
\netextent(T, \lambda) \geq N^2_-(T) + N^2_+(T) + \max |\boundary S|.
\]
Furthermore, for any ditree $T$, there exists a productless positive labelling $\lambda$ such that equality holds.
\end{theorem}

\begin{figure}[ht!]
\centering
\labellist
\small\hair 2pt
\pinlabel {$2$} [b] at 7 197
\pinlabel {$2$} [b] at 79 197
\pinlabel {$2$} [b] at 141 197
\pinlabel {$2$} [b] at 214 197
\pinlabel {$2$} [tl] at 121 145
\pinlabel {$2$} [bl] at 118 81
\pinlabel {$2$} [l] at 18 9
\pinlabel {$2$} [r] at 102 9
\pinlabel {$2$} [r] at 205 9
\pinlabel {$1$} [r] at 6 102
\pinlabel {$1$} [r] at 106 116
\pinlabel {$1$} [tr] at 52 135
\pinlabel {$1$} [bl] at 95 172
\pinlabel {$1$} [br] at 130 172
\pinlabel {$1$} [l] at 217 135
\pinlabel {$1$} [r] at 109 51
\pinlabel {$1$} [bl] at 162 47
\endlabellist
\includegraphics[scale=0.5]{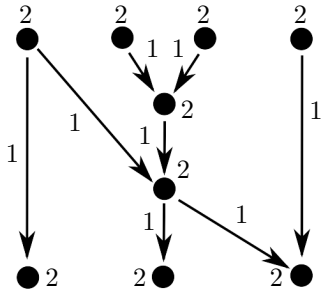}
\caption{A width tree $T$ achieving equality for the lower bound of Theorem \ref{minflowmaxcut}. In this case we have $N^2_-(T) = 1$; $N^2_+(T) = 4$, and $\max |\boundary S| = 5$. In this example, the set of sources is a strong source-sink cut realizing the maximum.}
\label{fig:widthtreeeq}
\end{figure}

The following lemma is a version of \cite[Equation (1)]{TT2}. It implies (among other things) that $\netextent(T, \lambda) \geq 0$ for nonnegative width trees. For a ditree $T$, with vertex $v$, we let $v_+$ denote the set of edges with tail at $v$ and $v_-$ the set of edges with head at $v$.

\begin{lemma}\label{easy sum}
Suppose $(T, \lambda)$ is a width tree with vertex set $V$. Then 
\[
\netextent(T, \lambda) = \sum\limits_{v \in V}(\lambda(v) - \lambda(v_-)) = \sum\limits_{v \in V}(\lambda(v) - \lambda(v_+)).
\]
\end{lemma}
\begin{proof}
Each edge is an incoming edge for some vertex and an outgoing edge for some vertex.
\end{proof}

\begin{definition}
Suppose that $T$ is a ditree.  A positive labelling  $F$ on $T$ is \defn{conservative} if it satisfies the \defn{flow constraints}; that is:
\begin{itemize}
\item for each vertex $v$ that is neither a source nor a sink, $F(v_-) = F(v) = F(v_+)$;
\item for each sink vertex $v$, $F(v) = F(v_-)$;
\item for each source vertex $v$, $F(v) = F(v_+)$. 
\end{itemize}
A conservative labelling is called a \defn{flow}. The \defn{value} of a flow $F$ is the total value of $F$ on the edges incoming to the sink vertices of $T$. The \defn{value} $\val(F, S)$ of a flow $F$ on a strong source-sink cut $S$ is the total value of $F$ on the edges of $T$ having their tail in $S$ and their head not in $S$.
\end{definition}

\begin{remark}
Observe that for a flow $F$ and strong source-sink cut $S$, $\val(F, S) \geq |\boundary S|$ since $F$ is positive.
\end{remark}

Note that given a flow $F$ on a ditree $T$, the pair $(T, F)$ is a width tree. We begin by converting width labels into flows. Because we want to keep track of product edges, we make the following definition. See Figure \ref{fig:augmenting}.

\begin{figure}[ht!]
\includegraphics[scale=0.5]{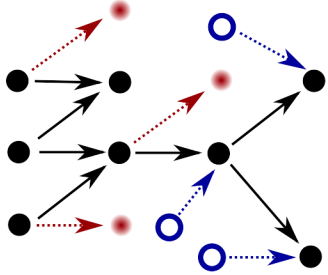}
\caption{Consider the ditree $T$ consisting of the solid black vertices and edges. Together with the hollow and fuzzy vertices and edges we have the augmented ditree $\wihat{T}$. (In the online version, the fuzzy vertices are red and the hollow vertices are blue, as are the incident edges.) The fuzzy vertices  and their incident edges are the $A_+$ and $E_+$ augmenting vertices and edges. The hollow vertices  and their incident edges are the $A_-$ and $E_-$ augmenting vertices and edges.}
\label{fig:augmenting}
\end{figure}

\begin{definition}
Suppose that $T$ is a ditree. For each vertex $v$ of $T$ that has a single outgoing edge, let $A_+(v)$ denote a vertex not in $T$ and let $E_+(v)$ denote an edge pointing from $v$ to $A_+(v)$. Similarly, for each vertex $v$ of $T$ that has a single incoming edge, let $A_-(v)$ denote a vertex not in $T$ and $E_-(v)$ an edge pointing from $A_-(v)$ to $A_+(v)$. We require that all the $A_\pm(v)$ vertices be distinct from each other. Let $\wihat{T}$ be the ditree obtained by taking the edges and vertices of $T$ together with all the $A_\pm(v)$ vertices and $E_\pm(v)$ edges. We call $\wihat{T}$ the \defn{augmented ditree} for $T$ and $E_\pm(v)$ and $A_\pm(v)$ the \defn{augmenting edges} and \defn{augmenting vertices}. Extend $\lambda$ to a labelling $\wihat{\lambda}$ on $\wihat{T}$ by declaring $\wihat{\lambda}$ to be one on all augmenting vertices and edges. Note that $(\wihat{T}, \wihat{\lambda})$ is a positive width tree where the only product edges are the augmenting edges.
\end{definition}

Figure \ref{fig:alg} shows an example of the first three steps of the algorithm used in the proof of the next theorem.

\begin{theorem}\label{flow exist}
Suppose that $(T, \lambda)$ is a positive width tree without product edges. Then there exists a flow $F$ on the augmented ditree $\wihat{T}$ such that $\netextent(\wihat{T}, F) = \netextent(T, \lambda)$. Furthermore, $F\geq \wihat{\lambda}$.
\end{theorem}

\begin{figure}[ht!]
\centering
\labellist
\small\hair 2pt
\pinlabel {$1$} [b] at 8 335
\pinlabel {$10$} [b] at 44 248
\pinlabel {$14$} [b] at 67 193
\pinlabel {1} [b] at 48 328
\pinlabel {1} [br] at 78 259
\pinlabel {6} [bl] at 93 236
\pinlabel {3} [br] at 101 202
\pinlabel {10} [b] at 116 181
\pinlabel{5} [b] at 109 335
\pinlabel{12} [l] at 159 232
\pinlabel{15} [b] at 170 188
\pinlabel{5} [l] at 243 164
\pinlabel{19} [tr] at 210 69
\pinlabel{7} [l] at 256 11
\pinlabel{3} [l] at 240 44
\pinlabel {$1$} [b] at 386 335 %
\pinlabel {$10$} [b] at 421 248 %
\pinlabel {$14$} [b] at 443 193 %

\pinlabel {1} [b] at 429 328
\pinlabel {1} [br] at 455 261
\pinlabel {9} [bl] at 471 236
\pinlabel {3} [br] at 477 205 
\pinlabel {10} [b] at 493 182

\pinlabel{5} [b] at 486 336
\pinlabel{15} [l] at 538 233 %
\pinlabel{15} [b] at 549 189
\pinlabel{5} [l] at 620 164
\pinlabel{19} [tr] at 589 69
\pinlabel{7} [l] at 633 11
\pinlabel{3} [l] at 616 44
\pinlabel {$1$} [b] at 751 335 
\pinlabel {$10$} [b] at  787 249 
\pinlabel {$14$} [b] at  809 196

\pinlabel {1} [b] at 794 331
\pinlabel {1} [br] at 822 263
\pinlabel {9} [bl] at 835 239
\pinlabel {3} [br] at 843 207
\pinlabel {11} [b] at 859 184

\pinlabel{5} [b] at 851 340
\pinlabel{15} [l] at 902 235
\pinlabel{16} [b] at 915 191
\pinlabel{6} [l] at 986 164
\pinlabel{20} [tr] at 956 73
\pinlabel{8} [l] at 996 15
\pinlabel{4} [l] at 979 50
\endlabellist
\includegraphics[scale=0.45]{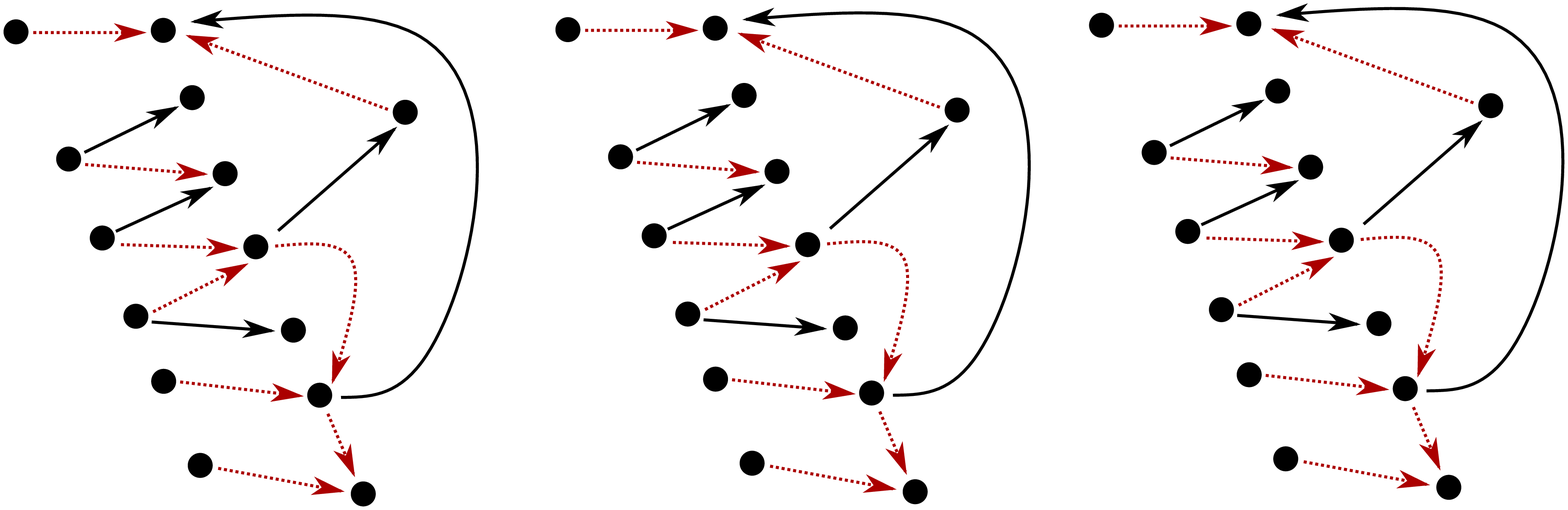}
\caption{The first three steps of the algorithm to convert an arbitrary labelling into a flow. We have not shown the labels on vertices and edges irrelevant to the first three steps. The vertices are ordered from left to right and the ``special edge'' departing from each vertex is dashed (red in the online version). The first step leaves the labelling unchanged.}
\label{fig:alg}
\end{figure}

\begin{proof}
Our strategy for constructing a flow is to increase label values on edges and internal vertices so as to push excess label values out to the sources and sinks of the ditree. Suppose that $(T, \lambda)$ is a positive width tree without product edges. Let $\wihat{S}$ be the set of sources of $\wihat{T}$.

Let $E_k$ be the subgraph of $\wihat{T}$ which consists of the edges and vertices of $\wihat{T}$ lying on a path starting at a vertex of $\wihat{S}$ and of length $k$. Order the vertices of $\wihat{T}$ as $v_1, \hdots, v_n$ so that for each $k$, all vertices of $E_k$ appear before the vertices of $E_{k+1}\setminus E_k$. For each vertex $v_i$, make a choice of outgoing edge, called the \defn{special edge} at $v_i$ and denoted $e(v_i)$. 

We define $F$ recursively. Let $\lambda_0 = \wihat{\lambda}$. Suppose we have defined $\lambda_i$ for $i < n$ so that $(\wihat{T},\lambda_i)$ is a width tree. We define $\lambda_{i+1}$ as follows. Let $\Delta = \lambda_i(v_{i+1}) - \lambda_i((v_{i+1})_+)$ be the difference between the label at $v_i$ and the sum of the labels on edges leaving $v_i$. Let $Q$ be the set of maximal coherent paths in $\wihat{T}$ that start at $v_{i+1}$ and traverse only special edges in $\wihat{T}$. (Maximal means that no path can be extended to a longer coherent path traversing only special edges.) Let $\lambda_{i+1}$ be the labelling obtained by increasing the value of $\lambda$ on all vertices (apart from $v_{i+1}$) and edges of the path in $Q$ beginning at $v_{i+1}$ by $\Delta$, and leaving all other labels unchanged. 

Notice that any time we have increased the label of an edge in $\wihat{T}$, we have also increased the label of its terminal vertex by the same amount, so $(\wihat{T}, \lambda_{i+1})$ is still a width tree. Furthermore, if we increased the label of a nonsink vertex $v$, we also increased the value of exactly one incoming edge by the same amount. Consequently, $\netextent(\wihat{T}, \lambda_{i+1}) = \netextent(\wihat{T}, \lambda_i)$.  We have arranged for the label on $v_{i+1}$ to be equal to the total label on the outgoing edges from $v_{i+1}$. For each $j < i+1$ such that $v_j$ is not a sink vertex, the value of $\lambda_{i+1}$ on $v_j$ is equal to its total value on the outgoing edges from $v_j$. By recursion we arrive at a width tree $(\wihat{T}, \lambda_n)$ with the property that for each nonsink vertex $v$, $\lambda_n(v) = \lambda_n(v_+)$ and for every vertex $v$, $\lambda_n(v) - \lambda_n(v_-) = \wihat{\lambda}(v) - \wihat{\lambda}(v_-)$. 

Let $(\ob{T}, \lambda_n)$ be the ditree obtained by reversing the directions on all edges of $\wihat{T}$. Apply the previous algorithm to arrive at $(\ob{T}, \ob{\lambda})$. Then $\ob{\lambda}$ is flow on both $\wihat{T}$ and $\ob{T}$ with the same net extent as $\lambda$ and for every edge $e$ of $\wihat{T}$, $F(e) \geq \lambda(e)$.
\end{proof}

\begin{proof}[Proof of Theorem \ref{minflowmaxcut}]
Suppose that $(T, \lambda)$ is a productless positive width tree. Let  $\wihat{T}$ be the augmented ditree and let $F$ be the flow constructed by Theorem \ref{flow exist}. Notice that we can construct a natural bijection between the set of strong source-sink cuts for $T$ and the set of strong source-sink cuts for $\wihat{T}$: Given a source-sink cut $S$ for $T$ we include the source augmenting vertices, but not the sink augmenting vertices, into $S$ to arrive at a strong source-sink cut $\wihat{S}$ for $\wihat{T}$. Conversely, given a strong source-sink cut $\wihat{S}$ for $\wihat{T}$, we remove the augmenting vertices to arrive at a strong source-sink cut $S$ for $T$. Observe that $|\boundary \wihat{S}|$ equals the sum of  $|\boundary S|$ with $N^2_-(T)$ and $N^2_+(T)$. 

Let $\wihat{S}$ be any strong source-sink cut for $\wihat{T}$. Since $F$ is conservative, by Lemma \ref{easy sum}, $\netextent(\wihat{T}, \wihat{\lambda}) = \val(F; \wihat{S}) \geq |\boundary \wihat{S}|$. It remains to show that we can construct a flow achieving equality.

Let $T$ be a ditree. There exists a positive labelling $\lambda$ for $T$ such that $(T, \lambda)$ is a width tree without product edges. To see this, let $\lambda(e) = 1$ for every edge $e$ and let $\lambda(v) = \max \{\deg_-(v), \deg_+(v)\} + 1$ for every vertex $v$. From $\lambda$, we can construct a flow $F$ on $\wihat{T}$ using Theorem \ref{flow exist}. We now modify $F$ to construct a flow achieving equality.

An \defn{adjusting path} for $F$ is a possibly non-coherent path $\alpha$ from a source vertex to a sink vertex such that for each edge $e$ traversed by $\alpha$ in the forward direction, $F(e) \geq 2$. Let $\epsilon = \min F(e) - 1$ where the minimum is taken over all edges traversed by $\alpha$ in the forward direction. Observe $\epsilon > 0$. If $\alpha$ traverses $e$ in the forward direction, let $F'(e) = F(e) - \epsilon$. If $\alpha$ traverses $e$ in the backward direction, let $F'(e) = F(e) + \epsilon$. If $e$ is an edge not traversed by $\alpha$, let $F'(e) = F(e)$. For each vertex $v$ of $\wihat{T}$, let $F'(v) = \max\{F'(v_-), F'(v_+)\}$.  Observe that $(T, F')$ is a conservative flow and that 
\[
\netextent(\wihat{T}, F') = \netextent(\wihat{T}, F) - \epsilon.
\]
Consequently, we cannot create adjusted flows indefinitely.

We implement a version of the Ford-Fulkerson algorithm as in \cite{West}. Given a flow $F$ on $\wihat{T}$, the algorithm produces a positive flow $F'$ whose value (equivalently, net extent) is at most the value of $F$ and a strong source-sink cut $\wihat{S}$ of $\wihat{T}$ with $|\boundary \wihat{S}| = \netextent(\wihat{T}, F')$. In the algorithm, we have a set $R_k$ of \emph{reached vertices} and a set $S_k \subset R_k$ of \defn{searched vertices}.

Let $R_0$ be the set of sources in $\wihat{T}$ and let $S_0 = \nil$. Assume we have $R_k$ and $S_k \subset R_k$. If $S_k = R_k$ the algorithm terminates. Suppose that there is a vertex $v \in R_k \setminus S_k$. For each outgoing edge $e$ from $v$ with $F(e) \geq 2$, add the head of $e$ to $R_k$. Record the fact that it was added upon arrival from $v$. For each incoming edge to $v$ add the tail of $v$ to $R_k$ and record that it was reached from $v$. Let $R_{k+1}$ be the set we obtained from our original $R_k$. Add $v$ to $S_k$ to obtain $S_{k+1}$. If we ever add a sink $w$  to $R_{k}$ in the process of obtaining $R_{k+1}$, we can, using our labels, trace backwards from $w$ to construct an adjusting path. From that we construct an adjusted flow and apply the algorithm to the adjusted flow. The algorithm thus terminates with a positive flow $F'$ such that $\netextent(\wihat{T}, F') \leq \netextent(\wihat{T}, F)$ and with sets $R_k = S_k$.

 Let $\wihat{S} = R_{k}$.  Each edge $e$ with tail in $\wihat{S}$ and head not in $\wihat{S}$ must have $F'(e) = 1$, as otherwise we would have continued the algorithm. There are no edges $e$ with head in $\wihat{S}$ and tail not in  $\wihat{S}$, as the algorithm adds all such vertices to the reached set. Thus, $\wihat{S}$ is a strong source-sink cut and
\[
\netextent(\wihat{T}, F') = F'^+(\wihat{S}) - F'^-(\wihat{S}) = F'^+(\wihat{S}) = |\boundary \wihat{S}|.
\]

Let $S'$ be a strong source-sink cut on $\wihat{T}$, maximizing $|\boundary S'|$. By our previous remarks, we have
\[
|\boundary \wihat{S}| = \netextent(\wihat{T}, F') \geq |\boundary S'| \geq |\boundary \wihat{S}|,
\]
so $|\boundary \wihat{S}|$ is itself maximal. 
\end{proof}

We can apply this to study the difference between width and bridge number.
\begin{corollary}\label{cor: difference}
Suppose that $L \subset S^3$ is a link with associated $(T, \lambda) \in \T_2(S^3, L)$ such that $\width_b(L) = \width(T, \lambda)$. Then
\[
\b(L) \geq N_-^2(T) + N_+^2(T) + \max |\boundary S|
\]
where the maximum is over all  strong source-sink cuts $S$ for $T$. Additionally, if $T$ has at least two edges and if $L$ admits no essential Conway sphere, then
\[
\width_b(L) - \b(L) \geq N_-^2(T) + N_+^2(T) + \max |\boundary S|.
\]
\end{corollary}

\begin{proof}
The first statement follows from Theorems \ref{amalg} and \ref{minflowmaxcut}. Suppose, therefore, that $T$ has at least two edges and that $L$ admits no essential Conway sphere. By Theorem \ref{thm: slim main}, $\lambda \geq 2$. Consequently, $\mu = 2\lambda^2 - \lambda \geq 1$. The pair $(T, \mu)$ is a productless positive width tree. Note that $\netextent(T, \mu) = \netextent(T, 2\lambda^2) - \netextent(T, \lambda)$, so 
\[
\width_b(L) - \b(L) = \netextent(T, \mu) \geq N_-^2(T) + N_+^2(T) + \max |\boundary S|,
\]
as desired. 
\end{proof}

\section{Constructing associated knots}
In this section, we show:
\begin{theorem}\label{creation}
Suppose that $(T, \lambda, \delta)$ is a positive, productless width tree with distance threshold. There exists an even tangle $(M', \tau')$ and $\mc{H}' \in \H(M',\tau')$ associated to $(T, \lambda, \delta)$. If $T$ has no boundary vertices, then we may assume $M' = S^3$ and that $\tau'$ is a knot.
\end{theorem}

\begin{definition}
Let $\Lambda = \{ n \in \Z : n \geq -1\}$. A \defn{trivialpod} with labels $\lambda$ is a rooted ditree $T$ such that the following hold:
\begin{enumerate}
\item  The root is designated as the \defn{thick vertex} and all other vertices are designated as a \defn{thin vertices};
\item Each thin vertex is joined to the thick vertex by an edge and there are no edges with both endpoints at thin vertices;
\item Either all edges are oriented into the thick vertex or all edges are oriented out of the thick vertex;
\item Each vertex $v$ of $T$ has a label $\lambda(v) \in \Lambda$;
\item The label of the thick vertex is at least the sum of the \emph{nonnegative} labels of the thin vertices.
\end{enumerate}

Two trivialpods are equivalent if there is a label-preserving digraph isomorphism (possibly orientation-reversing) between them.
\end{definition}

Suppose $(C, T_C)$ is a c-trivial tangle where $\boundary_+ C$ has been transversally oriented either into or out of $C$ and each component of $\boundary_- C$ has also been given a transverse orientation which points out of or into $C$, respectively. We can then construct a trivialpod by declaring $\boundary_+ C$ to be the thick vertex and each component of $\boundary_- C$ to be a thin vertex and joining each thin vertex to the thick vertex by an edge which has been given an orientation corresponding to the orientation of $\boundary C$. For a vertex $S \cpt \boundary M$ of this trivialpod, we declare the label $\lambda(S)$ to be $-1 + |S \cap T_C|/2$. We say that the trivialpod and the c-trivial tangle $(C, T_C)$ are \defn{associated}. The next lemma shows that this concept is well-defined. 

\begin{lemma}
Association is a surjection $f$ from the set of even c-trivial tangles up to equivalence to the set of trivialpods up to equivalence, such that each even c-trivial tangle $(C, \tau)$ is associated to the trivialpod $f(C,\tau)$.
\end{lemma}
\begin{proof}
We start by showing that association is well-defined. That is, given a c-trivial tangle $(C, \tau)$ with $\boundary C$ transversaly oriented as above, the ditree $T$  with thick vertex $v$ and labels $\lambda$ constructed above is a trivialpod. The only condition that is not immediate is condition (5). To see that this holds, choose a collection of pairwise disjoint sc-discs $\Delta$ for $\boundary_+ C$ such that $(C, \tau) \setminus \Delta$ is the disjoint union of product tangles and elementary ball tangles. Each $\boundary$-reduction of a punctured surface $S$ along an unpunctured disc decreases $\extent(S) = -\chi(S) + |S \cap \tau|/2$ by two. Each decomposition along a once-punctured disc leaves $\extent(S)$ unchanged. At the conclusion of the $\boundary$-reductions using $\Delta$, the result holds and since the $\boundary$-reductions left $\boundary_- C$ alone, it also holds beforehand. This association clearly preserves equivalence.

We now show that $f$ is a surjection. Suppose that $T$ is a trivialpod with labels $\lambda$. Let $v$ be the thick vertex of $T$ and set $M = \lambda(v)$. We induct on the number of thin vertices. If $T$ has no thin vertices,  let $(C, \tau)$ be the result of inserting $M + 1$ bridge arcs into a 3-ball.  Observe that $\extent(\boundary_+ C) = -1 + (M + 1) = \lambda(v)$. Thus, the result holds in this case. Suppose, therefore, that there exists an $N \geq 0$ such that whenever $T'$ is a trivialpod having thick vertex $v$ with label $M$ and $0 \leq k \leq N$ thin vertices, then $T'$ is associated to some even c-trivial tangle.

Assume $T$ has $N+1$ thin vertices $v_1, \hdots, v_{N+1}$. Let $T'$ be the tree obtained by removing $v_{N+1}$ and its incident edge from $T$. The vertex $v$ is still a thick vertex for $T'$ and $T'$ inherits edge orientations and labels from $T$.  By our inductive hypothesis, there is an even c-trivial tangle $(C', \tau')$ such that $T'$ is associated to $(C', \tau')$. 

Let $S_i$ for $i \in \{1, \hdots, N\}$ be the components of $\boundary_- C'$ corresponding to thin vertex $v_i$. Let $C$ be the result of removing an open 3-ball from $C'$ that is disjoint from $\tau'$ and let $S_{N+1} = \boundary C \setminus \boundary C'$. As in the previous paragraph, $(C, \tau')$ is a c-trivial tangle with $S_1, \hdots, S_{N+1} \cpt \boundary_- C$.  Out of all c-tangles $(C, \tau'')$ having the properties that:
\begin{enumerate}
\item $\extent(\boundary_+ C) = \lambda(v)$,
\item For $i \leq N$, $\extent(S_i) = \lambda(v_i)$,
\item $\extent(S_{N+1}) \leq \lambda(v_{N+1})$
\end{enumerate}
assume that we have chosen $\tau''$ to minimize $\lambda(v_{N+1}) - \extent(S_{N+1})$. If we have equality, then set $(C, \tau) = (C, \tau'')$ and we are done. So suppose, for a contradiction, that $\lambda(v_{N+1}) > \extent(S_{N+1})$. In particular, observe that this means that $\lambda(v_{N+1}) \geq 0$.

If $\tau''$ contains a bridge arc $b$, let $\tau'''$ be the result of removing $b$ from $\tau''$ and inserting two vertical arcs, each joining $\boundary_+ C$ to $S_{N+1}$. This preserves the extent of $\boundary_+ C$, as well as $\extent(S_i)$ for $i \leq N$ but reduces $\lambda(v_{N+1}) - \extent(S_{N+1})$ by 1. Thus, $\tau''$ has no bridge arcs. If $\tau''$ contains a ghost arc $g$ having endpoints $S_k$ and $S_\ell$ for distinct $k, \ell \leq N$, then we can let $\tau'''$ be the result of replacing $g$ in $\tau''$ with ghost arcs $g_k$ and $g_\ell$ such that $g_k$ joins $S_k$ to $S_{N+1}$ and $g_\ell$ joins $S_\ell$ to $S_{N+1}$. Observe that $(C, \tau''')$ is still a c-trivial tangle but that $\lambda(v_{N+1}) - \extent(S_{N+1})$ has decreased by 1. Additionally, $\extent(\boundary_+ C)$ and $\extent(S_i)$ for $i \leq N$ remain unchanged. Thus, $\tau'''$ contradicts our choice of $\tau''$. Hence, every ghost arc in $\tau''$ has an endpoint on $S_{N+1}$. By the definition of c-trivial tangle, no ghost arc has both endpoints on $S_{N+1}$. Thus, we may assume that $\tau''$ contains only vertical arcs and ghost arcs with a single endpoint on $S_{N+1}$. 

The ghost arc graph $\Gamma$ of $(C, \tau'')$ has at most one component containing an edge and that component, if it exists, has $S_{N+1}$ as a vertex. We have remarked that $\lambda(v) = |\boundary_+ C \cap \tau''|/2 - 1 \geq 0$. Thus, the number of vertical arcs is $|\boundary_+ C \cap \tau''| \geq 2$. Suppose there exist distinct $j, k \leq N$ such that $S_j$ and $S_k$ are incident to vertical arcs $v_j$ and $v_k$ respectively. Let $\tau'''$ be the result of replacing $v_j$ and $v_k$ with ghost arcs $g_j$ and $g_k$ and bridge arc $b$. The ghost arc $g_j$ joins $S_j$ to $S_{N+1}$ and the ghost arc $g_{k}$ joins $S_k$ to $S_{N+1}$. This can be done so that $(C, \tau''')$ is a c-trivial tangle. We have preserved $\extent(\boundary_+ C)$ and $\extent(S_i)$ for $i \leq N$, but reduced $\lambda(v_{N+1}) - \extent(S_{N+1})$, a contradiction. Thus, all vertical arcs have one endpoint on a unique component $S_i$ for $i \leq N$. 

If the ghost arc graph for $(C, \tau'')$ contained an edge, then, as it is acyclic, it would contain at least two vertices of degree 1. Since $(C, \tau'')$ is an even c-trivial tangle, both of the corresponding components of $\boundary_- C$ would be incident to vertical arcs. Thus, the ghost arc graph consists of isolated vertices. Remove one of the vertical arcs having an endpoint at $S_i$ and replace it with a ghost arc joining $S_i$ and $S_{N+1}$ and a vertical arc joining $\boundary_+ C$ and $S_{N+1}$. Let $\tau$ be the new tangle. This can be done so that $(C, \tau)$ is a c-trivial tangle. For each $j \leq N$, we have $|S_j \cap \tau''| = |S_j \cap \tau|$ and we also have $|\boundary_+ C \cap \tau''| = |\boundary_+ C \cap \tau|$, but $\extent(S_{N+1}) - \lambda(v_{N+1})$ has decreased by 2, a contradiction. Thus, the association of even c-trivial tangles with trivialpods is surjective.
\end{proof}

The proof of the next lemma is easy and admits a great deal of latitude. We will subsequently make use of the latitude.
\begin{lemma}\label{existence 1}
Suppose that $(T,\lambda)$ is a width tree. Then there exists an even tangle $(M, \tau)$ such that $(T, \lambda)$ is associated to $(M,\tau)$. If $T$ has no boundary vertices and if $\lambda \geq 0$, then we may take $M = S^3$ and $\tau$ to be a knot.
\end{lemma}
\begin{proof}
Let $(T, \lambda)$ be a given width tree. For each nonboundary vertex $v \in T$, let $N_\pm(v)$ be the union of the vertex $v$ with ``half-edges'' of the outgoing/incoming edges of $T$ incident to $v$, respectively. Then each of $N_\pm(v)$ is a trivialpod, with $v$ as the thick vertex.  The labels on the thin vertices are inherited from the edges of $T$ incident to $v$. Choose a c-trivial tangle $(C_\pm(v), \tau_\pm(v))$ associated to each of these trivialpods. If $N_+(v)$ (for example) has no edges, then $\boundary_- C_+(v) = \nil$. Since $\boundary_+ C_-(v)$ and $\boundary_+ C_+(v)$ are each spheres with $2\lambda(v) + 2$ punctures, there is a homeomorphism $\phi_v$ taking one to the other and preserving punctures. Suppose that vertices $v$ and $w$ are adjacent via an edge $e$. Let $F \cpt \boundary_- C(v)$ and $F' \cpt \boundary_- C(w)$ be the components corresponding to $e$. Since $F$ and $F'$ are both spheres with $2\lambda(e) + 2$ punctures, there is a homeomorphism $\phi_e$ from $F$ to $F'$ taking punctures to punctures. Thus, we may glue all of the trivialpods together using the homeomorphisms $\phi_v$ and $\phi_e$ for all nonboundary vertices and edges of $T$. The result is a tangle $(M,\tau)$. Let $\mc{H}^+$ be the union of the surfaces $\boundary_+ C_\pm(v)$ in $M$. Let $\mc{H}^-$ be the union of the components of $\boundary_- C_\pm(v)$ that do not correspond to edges of $T$ incident to boundary vertices. Then $\mc{H} = \mc{H}^+ \cup \mc{H}^- $ is a multiple c-bridge surface for $(M,\tau)$ associated to $(T, \lambda)$. 

If $T$ has no boundary vertices, then $M = S^3$ and $\tau$ is a link.  If $\lambda \geq 0$ and $|\tau| \geq 2$, there exists a c-trivial tangle $(C, \tau_C) \cpt \mc{H}$ intersecting distinct components of $\tau$.  Each leaf and isolated vertex of the ghost arc graph is incident to vertical arcs, so some component $H$ of $\boundary C$ is incident to distinct components of $\tau$. Modify the gluing map along $H$ by a half Dehn twist to reduce the number of components of $\tau$. Consequently, we may perform the construction so that $\tau$ is a knot.
\end{proof}

\begin{definition}
Suppose that $H$ is a c-bridge surface for an irreducible $(M, \tau)$ with $|H \cap \tau| \geq 5$. The \defn{bridge distance} $d(H)$, the \defn{c-distance} $d_c(H)$, and the \defn{AD-distance} $d_{AD}(H)$ are each the minimum $n$ such that $d(\boundary A, \boundary B) = n$, if the minimum exists. For $d(H)$, we minimize over all compressing discs $A$ and $B$ for $H$ on opposite sides of $H$; for $d_c(H)$, we minimize over all c-discs $A$ and $B$ for $H$ on opposite sides of $H$; and for $d_{AD}(H)$, we minimize over all c-discs and vertical annuli $A$ and $B$ on opposite sides of $H$.
\end{definition}

\begin{remark}
In all cases with $|H \cap \tau| \geq 5$, $d_{AD}(H)$ is defined. If $H$ has a compressing disc on both sides, then $d(H)$ is defined. If $H$ has a c-disc to both sides, then $d_c(H)$ is defined. When $\boundary M$ has no spheres with two or fewer punctures, $d_c(H) = d_{sc}(H)$, as $H$ has no semi-cut discs or semi-compressing discs. See \cite[Lemma 3.5]{TT1}. 
\end{remark}

A simple outermost arc argument (as in \cite[Prop. 4.1]{Tomova}) allows a comparison between the three distances.

\begin{lemma}\label{comparing distances}
Suppose $(C, \tau)$ is a c-trivial tangle such that $\boundary_- C$ contains no spheres with two or fewer punctures. If $\boundary_+ C$ has a compressing disc in $(C, \tau)$, then the boundary of each cut disc is disjoint from the boundary of some compressing disc for $\boundary_+ C$. Likewise, if $\boundary_+ C$ has a c-disc, then the boundary component of every vertical annulus in $(C,\tau)$ in $H$ is disjoint from the boundary of some c-disc. Consequently, if $H$ has a compressing disc to both sides, then
\[
d(H) \geq d_c(H) \geq d(H) - 2.
\]
Likewise, if $H$ has a c-disc to both sides, then
\[
d_c(H) \geq d_{AD}(H) \geq d_c(H) - 2.
\]
\end{lemma}

There are numerous ways to create knots with high distance bridge spheres \cites{BTY, IJK, IS, JM}. We'll use a method involving crossing changes achieved by Dehn twists in a bridge surface. Suppose that $(M, \tau)$ is an even, irreducible tangle with $H \subset (M, \tau)$ an oriented c-bridge surface. Assume that every component of $\boundary M$ has at least four punctures and that no tangle on either side of $H$ is a product tangle or trivial ball tangle. Assume $H$ has at least 6 punctures. For an oriented curve $c \subset H$, let $(M(c), \tau(c))$ denote the result of composing the gluing map with a right Dehn twist around $c$. Observe that $H$ is still a bridge surface for $(M(c), \tau(c))$ with the same extent. Blair, Campisi, Johnson, Taylor, and Tomova \cite{MHL2} prove the following theorem:

\begin{theorem}\label{MHL2}
Suppose that $H$ is a bridge sphere for a link $L \subset S^3$ such that $|H \cap L| \geq 6$. For any $N \in \N$, there exists a simple closed curve $c \subset H$ bounding a twice-punctured disc in $H$ such that if the gluing map along $H$ is composed with a Dehn twist along $c$ to create the link $L' \subset S^3$, then as a bride surface for $L'$, $d(H) \geq N$.
\end{theorem}

Here is the generalization we need:
\begin{theorem}\label{MHL-redux}
Suppose that $(M, \tau)$ is an irreducible even tangle such that every component of $\boundary M$ has at least four punctures. Suppose also that $H  \in \H(M, \tau)$ is a c-bridge sphere with at least 6 punctures and that neither component of $(M, \tau) \setminus H$ is a product tangle. Then for any $N \geq \N$ there exists a curve $c \subset H$ bounding a twice-punctured disc, such that after performing a Dehn twist along $c$, the c-bridge surface $H$ has $d_{AD}(H) \geq N$. 
\end{theorem}

\begin{remark}
Notice that performing the twist does not change the underlying homeomorphism type of the 3-manifold, since $c$ bounds a disc in $M$.  Also observe that on each side of $H$, the new tangle $(M', \tau')$ has the same number of vertical arcs, bridge arcs, and ghost arcs as the corresponding side of $H$ in $(M, \tau)$.
\end{remark}

\begin{proof}[Proof of Theorem \ref{MHL-redux}]
Let $N \in \N$. Observe that $H$ has a c-disc on each side as $|H \cap \tau| \geq 6$ and neither component of  $(M, \tau) \setminus H$ is a product tangle.

We begin by constructing a link $L$ from $\tau$. Since each component of $\boundary M$ intersects $\tau$ an even number of times, we may cap each component $S \subset \boundary M$ off with a tangle $(B(S), \tau(S))$ where $B(S)$ is a 3-ball and $\tau(S)$ is a collection of $|S \cap \tau|/2$ boundary-parallel arcs. This converts $\tau$ into a link $L$ and $M$ into the 3-sphere. The sphere $H$ is easily seen to be a bridge sphere for $L$: complete collections of sc-discs for each side of $H$ in $(M,\tau)$ extend to complete collections of sc-discs for each side of $H$ in $(S^3, L)$. Consequently, $H$ is a c-bridge sphere for $L$. Since $\boundary S^3 = \nil$, there are no ghost arcs, so it is actually a bridge sphere.

By Theorem \ref{MHL2}, there is a simple closed curve $c \subset H$, bounding a twice punctured disc in $H$ such that if $L'$ is the link resulting from a Dehn twist of $H$ around $c$, then $d(H) \geq N + 4$, when considering $H$ as a bridge surface for $L'$. By Lemma \ref{comparing distances}, we have $d_{c}(L') \geq N+2$. Let $(M', \tau')$ be the result of performing the twist in $(M,\tau)$. The tangle $(M', \tau')$ is also the result of removing the trivial tangles $(B(S), \tau(S))$ from each component $S$ of $\boundary M \subset M'$. Each c-disc for $H$ in $(M', \tau')\setminus H$ is also a c-disc for $H$ in $(S^3, L)\setminus H$ since no component of $\boundary M$ has two or fewer punctures. Since neither component of $(M,\tau)\setminus H$ is a product tangle, there is at least one such c-disc to each side of $H$ in $(M', \tau')$. Consequently, $d_{c}(H) \geq N+2$, when considering $H$ as a c-bridge sphere for $(M', \tau')$.  By Lemma \ref{comparing distances}, $d_{AD}(H) \geq N$, when considering $H$ as a c-bridge sphere for $H$ in $(M', \tau')$.
\end{proof}

To deal with the four-punctured spheres, we need the following standard result.

\begin{lemma}\label{2-bridge}
Let $H$ be a 2-sphere in $S^3$. For each $N \in \N$, there exists a knot $K$ such that $H$ is a bridge sphere for $H$, $|H \cap K| = 4$, and $d(H) = d_c(H) \geq N$.
\end{lemma}
\begin{proof}
The set of isotopy classes of essential curves on a 4-punctured sphere are the vertices of the Farey graph and two vertices are joined by an edge if they can be represented by curves with intersecting minimally exactly twice. The lemma follows from the fact that the Farey graph has infinite diameter and the fact that each c-disc for a bridge sphere with four punctures is a compressing disc.
\end{proof}

\begin{proof}[Proof of Theorem \ref{creation}]
Let $(T, \lambda, \delta)$ be a positive productless width tree with distance threshold. By Theorem \ref{existence 1}, there exists a tangle $(M, \tau)$ with $\mc{H}_1 \in \H(M, \tau)$ associated to $(T, \lambda)$. If $T$ is boundaryless, then we may assume that $\tau$ is a knot.  Suppose that some nonboundary vertex $v$ of $T$ has $\lambda(v) \leq 1$. Since $T$ is productless and positive, if $e$ is an edge incident to $v$, then $\lambda(e) \leq 0$, contradicting the assumption that $T$ is positive. Thus, either $T$ is a single vertex, or every thick surface $H' \cpt \mc{H}_1^+$ has $|H' \cap \tau| \geq 6$. In the former case, we appeal to Lemma \ref{2-bridge}. Suppose, therefore, that we are in the latter case.

Suppose $H' \cpt \mc{H}^+_1$. Let $(M_1, \tau_1)$ be the component of $(M, \tau)\setminus \mc{H}_1^-$ containing it. Since $T$ has no product edges, no component of $(M_1, \tau_1) \setminus H'$ is a product tangle. By Theorem \ref{MHL-redux}, there is an essential simple closed curve $c_{H'} \subset H'$ so that Dehn twisting along $c_{H'}$ converts $(M_1, \tau_1)$ into a tangle $(M'_1, \tau'_1)$ and $H'$ into a c-bridge surface $H' \in \H(M'_2, \tau'_2)$ with $d_c(H') \geq \delta$. If $\tau$ is a knot, the Dehn twist does not change that. Do this for each component $H' \cpt \mc{H}_1^+$. The result is then a multiple c-bridge surface $\mc{H}'$ for a tangle $(M', \tau')$ such that $d_c(H') \geq \delta$ for each $H' \cpt \mc{H}'^+$. If $T$ has no boundary vertices, then $M = M' = S^3$ and we may assume that $\tau'$ is a knot.
\end{proof}

\section{Uniqueness of width trees with small labels and high distance threshold}
In this section, we prove:

\begin{theorem}\label{biggest one}
Let $(M,\tau)$ be an even tangle such that each component of $\boundary M$ has at least four punctures. Let $(T', \lambda')$ be a positive slim width tree associated to $(M,\tau)$ such that $\netextent(T',\lambda') = \netextent_{-2}(M,\tau)$ or $\width(T', \lambda') = \width_{-2}(M,\tau)$. Suppose that $(T, \lambda, \delta) \in \T_2(M,\tau)$ with $\delta > 2\tr(T, \lambda)$. Then $(T', \lambda')$ is equivalent to $(T, \lambda)$. Indeed any  $\mc{J} \in \H(M,\tau)$ associated to $(T', \lambda')$ is $\tau$-parallel to any $\mc{H} \in \H(M,\tau)$ associated to $(T, \lambda)$.
\end{theorem}

Assuming the theorem, we have the following corollaries.
\begin{corollary}\label{uniqueness}
Let $(T, \lambda)$ be a productless positive width tree. Then there exists a tangle $(M,\tau)$ such that $(T, \lambda)$ is the unique element of $\T_2(M,\tau)$ up to equivalence and
\[\begin{array}{rcl}
\netextent(T,\lambda) &=& \netextent_{-2}(M,\tau), \text{ and } \\
\width(T, \lambda) &=& \width_{-2}(M,\tau)
\end{array}.
\]
If $T$ is boundaryless, then $M = S^3$ and we can take $\tau$ to be a knot. 
\end{corollary}
\begin{proof}
Set $\delta = 2\netextent(T,\lambda) + 3$. Notice that $(T, \lambda, \delta)$ is slim. By Theorem \ref{creation}, there exists a tangle $(M,\tau)$ such that $(T, \lambda, \delta) \in \T_2(M,\tau)$. If $T$ is boundaryless then $M = S^3$ and we can take $\tau$ to be a knot. By Theorem \ref{biggest one}, any locally thin multiple c-bridge surface for $(M,\tau)$ realizing either $\netextent_{-2}(M,\tau)$ or $\width_{-2}(M,\tau)$ induces a width tree equivalent to $(T, \lambda)$.
\end{proof}

\begin{corollary}\label{cor: sharp}
For any ditree $T$, there exists a positive, productless labelling $\lambda$ and a tangle $(M,\tau)$ such that $(T,\lambda) \in \T_2(M,\tau)$ and so that the inequalities of Theorem \ref{minflowmaxcut} are equalities. If $T$ is boundaryless, then $M = S^3$ and $\tau$ can be taken to be a knot.
\end{corollary}
\begin{proof}
By Theorem \ref{minflowmaxcut}, there is a labelling making the inequalities equalities. The result follows from Corollary \ref{uniqueness}. 
\end{proof}

The remainder of this section is devoted to proving Theorem \ref{biggest one}. Assume the hypotheses in  the statement of Theorem \ref{biggest one}. By Lemma \ref{thm: slim main}, $(M,\tau)$ is irreducible and prime. If $(M,\tau)$ were a product tangle, by Lemma \ref{thm: slim main} there would be no associated slim width trees, so $(M,\tau)$ is not a product.  Our proof is closely modelled on \cite{JT} (see also \cite{BZ}), with modifications inspired by \cite{MHL1}. The basic idea is that when thick spheres have high distance relative to the number of punctures, thick spheres and thin spheres having fewer punctures can be isotoped off them. If $\mc{H}$ and $\mc{J}$ are two multiple c-bridge surfaces with $\mc{H}$ being high distance, we first apply the philosophy to show that unless the surfaces in $\mc{J}$ have a large number of punctures compared with $\mc{H}$, we can isotope the thin surfaces $\mc{J}^-$ into $\mc{H}^-$. We then apply the same philosophy to the thick surfaces $\mc{J}^+$ and show that this implies that either $\mc{J}$ has higher net extent or width or the width trees associated to $\mc{H}$ and $\mc{J}$ are equivalent and that this equivalence arises from an isotopy of $\mc{J}$ to $\mc{H}$.

We begin with a lemma that will simplify subsequent discussion.

\begin{lemma}
If $T$ has a nonboundary vertex with label equal to 1, then the conclusions of Theorem \ref{biggest one} hold.
\end{lemma}
\begin{proof}
Suppose that $T$ has a nonboundary vertex $v$ of label 1. Let $\mc{H} \in \H(M,\tau)$ be associated to $T$ and choose $H \cpt \mc{H}^+$ such that $H$ is a four-punctured sphere. Since $(T,\lambda)$ is productless and positive, any edge $e$ incident to $v$ must have $\lambda(e) \leq 0$, a contradiction. Since $T$ is positive and slim,  $\tau$ must be a 2-bridge knot or link in $M = S^3$. Since $(T', \lambda')$ is slim, $\mc{H}'$ is also necessarily a bridge sphere for $(M,\tau)$. By \cite[Theorem 4.3]{ST}, $\mc{H}^+$ and $\mc{J}^+$ are isotopic, as desired.  
\end{proof}

Henceforth, we assume that no vertex of $T$ has label 1; that is each component of $\mc{H}^+$ has at least six punctures. As with all arguments of this sort, sweepouts are the key tool. The height functions used in the introduction to define bridge number are examples of sweepouts. A \defn{spine} for a c-trivial tangle $(C, \tau_C)$ is the union of $\boundary_- C$ with a certain graph $\Gamma_0 \subset C$. The graph portion $\Gamma_0$ of a the spine is empty if and only if $(C, \tau_C)$ is a trivial product compressionbody.  The key property of a spine is that  $(C, \tau_C) \setminus \Gamma_0$ is pairwise homeomorphic to $(\boundary_+ C \times (0,1], \text{points} \times (0,1])$ by a homeomorphism taking $\boundary_+ C$ to $\boundary_+ C \times \{1\}$ and the points $\boundary_+ C \cap \tau_C$ to $\text{(points)} \times [0,1)$. For more detail regarding the definition of spine that we use, see \cite[Definition 7.1]{TT1}.

\begin{definition}
Suppose that $(M', \tau')$ is a tangle with a c-bridge surface $H$. A \defn{sweepout} of $(M',\tau')$ by $H$ is a continuous function
\[
\phi \co M \to I
\]
such that $\phi^{-1}(-1)$ is a spine for the c-trivial tangle below $H$, $\phi^{-1}(+1)$ is a spine for the c-trivial tangle above $H$, and for $t \in (-1,1)$, the inverse images of $t$ are closed surfaces $H_t = \phi^{-1}(t)$ giving a foliation of $M' \setminus \phi^{-1}(\boundary I)$ transverse to the bridge arcs and vertical arcs of $\tau'$. This foliation defines a \defn{canonical projection} (that is, homeomorphism) $\pi_t \co H_t \to H$. 
\end{definition}

It follows easily from the definition of spine that a sweepout by $H$ exists. 

\begin{definition}
Suppose that $S \subset (M',\tau')$ and that $H$ is a c-bridge surface for $(M',\tau')$. Suppose that $\phi$ is a sweepout by $H$ with the property $\phi|_S$ is a Morse function. Also assume it has the property that all but at most one critical value $s \in (-1,1)$ of $\phi|_S$ has one critical point in its preimage and if there is such an \defn{exceptional} critical value $s$, then $\phi^{-1}|_S(s)$ contains exactly two critical points. We say that $\phi$ is \defn{adapted to $S$} if:
\begin{enumerate}
\item[(A)] either there is no exceptional critical value or if there is, for regular values $t, t'$ separated by the exceptional critical value $s$ and no other critical values, the projections of the curves $S \cap H_t$ and $S \cap H_{t'}$ can be isotoped to be disjoint.
\end{enumerate}
\end{definition}

The next proposition gives the fundamental connection between sweepouts and distance. Although somewhat inconvenient, we do need to allow our Morse functions to have two critical points at the same height. This is used when we analyze the graphic that arises by comparing two sweepouts. See \cite{BS} or Section 4 of \cite{MHL1} for  more detailed proofs along the same lines.

\begin{proposition}\label{prop: distance bd}
Suppose that $(M', \tau')$ is a tangle such that $\boundary M'$ is c-incompressible. Suppose $H \subset (M',\tau')$ is an sc-strongly irreducible c-bridge surface and that $S \subset (M',\tau')$ is an orientable connected surface with essential boundary in $\boundary M'$. Suppose that $|H \cap \tau'| > 4$ and that $S$ is not an unpunctured sphere. If there exists a sweepout $\phi$ by $H$ adapted to $S$, then after a perturbation of $\phi$, one of the following occurs:
\begin{enumerate}
\item There exists a regular value $t$ such that each curve of $H_t \cap S$ is inessential in $H_t$,
\item There exists a trivial ball tangle $(B, B \cap \tau') \subset (M',\tau')$ such that $S \subset (B, B \cap \tau')$, or
\item  $d_{AD}(H) \leq 2\extent(S) + 2\iota$,
where $\iota = 0$ if there is no exceptional critical value and $\iota = 1$ if there is.
\end{enumerate}
\end{proposition}
\begin{proof}[Proof sketch]
Let $-1 = s_0 < s_1 < \hdots < s_n = 1$ be the critical values of $\phi|_S$. For $i \in \{0, \hdots, n-1\}$, let $I_i = [s_i, s_{i+1}]$ and let $\inter{I}_i$ be its interior. Let $\gamma_t = S \cap H_t$ and $\pi(\gamma_t) = \pi_t(\gamma_t)$ for all $t \in (-1,1)$.

We begin by analyzing the relationship between nearby regular values. It will turn out that regular values corresponding to saddle points are the most important. 

If $t, t' \in \inter{I}_i$, then curves $H_t \cap S$ and $H_{t'} \cap S$ have isotopic projections to $H$. For such $t, t'$, the each curve of $S \cap H_t$ bounds an annulus in $S$ with a curve of $S \cap H_{t'}$ and the image of this annulus under $\phi$ is the closed interval between $t$ and $t'$. 

Suppose that $t \in \inter{I}_i$ and $t' \in \inter{I}_{i+1}$ and that $s_i$ is not the exceptional critical value.  Let $p_i = \phi|_S^{-1}(s_i)$. The critical point $p_i$ is either a center or a saddle point. If it is a center, then $\pi(\gamma_{t'})$ is obtained from $\pi(\gamma_t)$ by an isotopy and either including or removing a component that is inessential on $H$. If it is a saddle point, then $\pi(\gamma_{t'})$ is obtained from $\pi(\gamma_t)$ by attaching a band with core $\kappa$ and then performing an isotopy. If the endpoints of $\kappa$ lie on the same component of $\pi(\gamma_t)$, then they approach that component from the same side, as both $S$ and $H$ are orientable. Consider the curves of $\pi(\gamma_t)$ incident to the endpoints of $\kappa$ as the ``old components'' and the curves of $\pi(\gamma_{t'})$ incident to the cocore of the band (dual to $\kappa$) as the ``new components.'' Observe that since $H$ is orientable, the curves $\pi(\gamma_{t'})$ can be isotoped to be disjoint from the curves $\pi(\gamma_t)$. If there are two old components, then there is one new component and vice-versa.  If there are two old components and both are essential in $H$, then the new component is also essential unless those components bound either an annulus or a once-punctured in annulus in $H$ containing $\kappa$. If there is one old component $\sigma$ and if at least one new component of $\pi(\gamma_{t'})$ is inessential in $H$, then there is a subarc of $\sigma$ such that $\kappa \cup \sigma$ bounds an unpunctured disc or once-punctured disc in $H$. In any case, in $S$, the components of $\gamma_t$ whose projections are the old components in $H$ and the components of $\gamma_{t'}$ whose projections are the new components in $H$, cobound a pair of pants in $S$. Each non-old curve of $\gamma_t$ bounds an annulus in $S$ with a non-new curve of $\gamma_{t'}$ and vice versa. 

Finally, suppose that $t \in \inter{I}_i$ and $t' \in \inter{I}_{i+1}$ and that $s_i$ is the exceptional critical value. Let $p_i$ and $p'_i$ be the critical points having critical value $s_i$. If one or both of $p_i$ or $p'_i$ is a center, then the analysis is as in the previous case. Suppose that both $p_i$ and $p'_i$ are saddle points. To obtain $\pi(\gamma_{t'})$ from $\pi(\gamma_{t})$, there are now two disjoint bands $\kappa_a$ and $\kappa_b$ in $H$ having endpoints on $\pi(\gamma_t)$ and interiors disjoint from $\pi(\gamma_t)$. The curves $\pi(\gamma_{t'})$ are obtained by attaching both bands and then an isotopy. By (A), the curves $\pi(\gamma_{t'})$ can be isotoped in $H$ to be disjoint from the curves $\pi(\gamma_t)$. 

For each $i \in \{0, \hdots, n-1\}$, choose a regular value $t_i \in I_i$. Label the interval $I_i$ with $\up$ (resp. $\down$) if, for $t \in \inter{I}_i$, if there is a curve of $S \cap H_t$ which is essential in $H_t$ which bounds an sc-disc or vertical annulus above (resp. below) $H_t$.  If some interval $I_i$ has both labels $\down$ and $\up$ or if $I_i$ and $I_{i+1}$  have opposite labels, then the projections of the curves giving rise to the labels show that $d_{AD}(H) \leq 1$. Thus no interval has both labels and no two adjacent intervals have opposite labels.

Let $G_\pm = \phi^{-1}(\pm 1)$ be the spines above and below $H$, determined by the sweepout $\phi$. They are each transverse to $S$.  If either one is disjoint from $S$ we have our first conclusion. Assume that neither is. Since $\phi^{-1}([0,t_{0}])$ is a small regular neighborhood of $G_-$, each curve (and there is at least one) of $\gamma_{t_0}$ is essential in $H$.  Let $\gamma \cpt \gamma_{t_0}$. The curve $\gamma$ either bounds a vertical annulus in the c-trivial tangle below $H_t$ with a component of $\boundary S$ or $\gamma$ is a meridian of an edge of $G_-$. In which case, it bounds a c-disc in $S$ that is below $H_t$. Thus,  $I_0$ is labelled $\down$. Similarly,  $I_{n-1}$ is labelled $\up$. It follows that there exist indices $k < \ell$, such that for each $k \leq i < \ell$, the interval $I_i$ is unlabelled, the interval $I_{k-1}$ has label $\down$ and the interval $I_\ell$ has label $\up$.  

For convenience, we rechoose the sweepout $\phi$ as folllows. Out of all sweepouts of $M'$ arising from $H$ and adapted to $S$ and isotopic to our original $\phi$, choose $\phi$ so that the set of critical points with critical values $\{s_k, \hdots, s_\ell\}$ (as above) contains as few saddles as possible. We note that an isotopy of $S$ can be reinterpreted as an isotopy of $\phi$. 

\textbf{Claim 1:} For $i \in \{k, \hdots, \ell-1\}$, each curve $\gamma$ of $\gamma_{t_i}$ that is essential in $H_{t_i}$  is also essential in $S$.

\emph{Proof of Claim 1:}  Suppose, to the contrary, that some $\gamma \cpt \gamma_{t_i}$ is essential in $H_{t_i}$ but inessential in $S$. It bounds an unpunctured or once-punctured disc in $S$.  Out of all unpunctured discs or once-punctured discs (not necessarily contained in $S$) with boundary $\gamma$, let $E$ be one which is transverse to $H_{t_i}$ and minimizes the number of components of intersection between the interior of $E$ and $H_{t_i}$. Since $I_i$ is unlabelled and $\gamma_i$ is essential in $H_{t_i}$, the disc $E$ must intersect $H_{t_i}$ in its interior. Let $\zeta$ be an innermost such curve. It must be inessential in $H_{t_i}$ and so, since $(M, \tau')$ is prime, we can isotope $E$ so as to remove $\zeta$, contradicting our choice of $E$. Thus, $\gamma$ is essential in $S$. \qed (Claim 1)

Assume that for all $i \in \{k, \hdots, \ell-1\}$ there exists a curve of $\gamma_{t_i}$ that is essential in $H_{t_i}$, for if this is not the case then Conclusion (1) holds. From each $H_{t_i}$ we can then choose a curve of $H_{t_i} \cap S$  essential in $S$ such that,  after eliminating adjacent repetitions, the isotopy classes of the projections of these curves give a path $\alpha$ between the annulus/disc sets of $H$. 

Let $W_i = \phi^{-1}([s_i, s_{i+1}])$. Recall that $W_i$ is homeomorphic to $H \times I$ and that $\kappa \cap W_i$  consists of vertical tangles. Say that a critical value $s_i \in \{s_k, \hdots, s_\ell\}$ is a \defn{useful saddle value} if one of the critical points of $\phi|_S$ with height $s_i$ is a saddle point involved with $\alpha_i$ and if for every such saddle point no new curve involved with the saddle point bounds an unpunctured disc in $H_{t_{i+1}}$. Let $N$ be the number of useful saddle values.

\textbf{Claim 2:} $d_{AD}(H) \leq N + 1$.

Suppose that $\pi(\alpha_i)$ and $\pi(\alpha_{i+1})$ are not isotopic in $H$. By construction, no curve of $\gamma_{t_{i+1}}$ has its projection to $H$ isotopic to $\pi(\alpha_i)$. Thus, in $\phi|_S^{-1}(s_{i+1})$ there is a saddle point $p$ such that $\alpha_i$ is involved with $p$. If $s_{i+1}$ is not an exceptional critical value then $s_{i+1}$ must be a useful saddle value, as otherwise there would be a curve of $\gamma_{t_{i+1}}$ whose projection to $H$ is isotopic to $\pi(\alpha_{i+1})$. Since there is at most one exceptional critical value, the length of $\alpha$ is at most $N+1$ and the result follows. \qed(Claim 2)

\textbf{Claim 3:} $N - 1\leq 2\extent(S)$.

Let $W_i = \phi^{-1}([t_i, t_{i+1}])$ for $i \in \{k, \hdots, \ell-1\}$. Consider a useful saddle value $s_i$ which is not an exceptional critical value. Let $S_i$ be the component of $S \cap W_i$ containing $\alpha_i$. It must be a pair-of-pants. By definition of useful saddle value, no component of $\boundary S_i$ bounds an unpunctured disc in $\boundary W_i$. Thus, if a component of $\boundary S_i$ is inessential in $S$, it is either $\boundary$-parallel or bounds a once-punctured disc. Let $S'$ be the union of all such $S_i$. Note that $-\chi(S')$ is either equal to $N$ or $N-1$ (depending on whether or not the exceptional critical value exists and is a useful saddle point). Let $S'' \cpt S \setminus S'$. It is not an unpunctured disc or sphere or once-punctured sphere, so $-\chi(S'') + |S'' \cap \tau'| \geq 0$. Thus, $N - 1 \leq -\chi(S) + |S \cap \tau'|$.\qed(Claim 3)

Consequently,
\[
d_{AD}(H) \leq 2\extent(S) + 2\iota,
\]
where $\iota = 0$ if there is no exceptional critical value and $\iota = 1$ if there is.
\end{proof}

We now resume the proof of Theorem \ref{biggest one}. 

\begin{proposition}\label{thin surface isotopy}
There is an isotopy of $\mc{J}$ taking $\mc{J}^-$ into $\mc{H}^-$.
\end{proposition}
\begin{proof}
By Theorem \ref{thm: slim main}, each component of $\mc{H}^-$ is c-incompressible and each component with two or fewer punctures is essential. Since $\lambda$ is positive, no component of $\mc{H}^-$ is a sphere with two or fewer punctures. Since both $\mc{H}^-$ and $\mc{J}^-$ are c-incompressible and $(M, \tau)$ is prime, we may isotope $\mc{J}$ so that all curves of $\mc{J}^-\cap \mc{H}^-$ are essential in both surfaces. Subject to that requirement, perform the isotopy so as to minimize $|\mc{J}^- \cap \mc{H}^-|$. Thus, no component of $\mc{J}^- \cap (M \setminus \mc{H}^-)$ with boundary is $\tau$-parallel to subsurface of $\mc{H}^-$. Let $F \cpt \mc{J}^- \setminus \mc{H}^-$ and let $(M', \tau') \cpt (M,\tau)\setminus \mc{H}^-$ be the component containing $F$. Let $F' \cpt \mc{J}^-$ be the component containing $F$. Since each component of $\mc{J}^- \cap \mc{H}^-$ is essential in both surfaces, $\extent(F) \leq \extent(F')$.  Let $H = \mc{H}^+ \cap M'$. Since $(T, \lambda)$ is slim and $(M,\tau)$ is prime and irreducible, $d_c(H) = d_{sc}(H) \geq 2$; in particular $H$ is sc-strongly irreducible.

Let $\phi$ be a sweepout of $(M', \tau')$ corresponding to $H$. We may perturb $\phi$ so that $\phi|_S$ is a Morse function with all critical points at distinct heights. (In particular, there is no exceptional critical value.) Hence, $\phi$ is adapted to $F$.  Let $t$ be a regular value. If $\gamma \cpt H_t \cap F$ is inessential in $H_t$, then it must also be inessential in $F$ as $F$ is c-incompressible. Consequently, if there exists a regular value $t$ such that each curve of $H_t \cap F$ is inessential in $H_t$, then $F$ can be isotoped, relative to its boundary, to be disjoint from $H$. Since it is then a c-incompressible surface in a c-trivial tangle, it must be $\tau$-parallel to a component of $\boundary M \cup \mc{H}^-$. In such a case, the surface $F$ must be closed and so $F$ is isotopic to a component of $\mc{H}^-$. Henceforth, assume there does not exist a regular value $t$ such that each curve of $H_t \cap F$ is inessential in $H_t$.

Since $F$ is c-incompressible and not an inessential sphere with two or fewer punctures, $F$ is not contained in a trivial ball tangle in $(M',\tau')$. Thus, by Proposition \ref{prop: distance bd} $d_{AD}(H) \leq 2\extent(S)$. By Lemma \ref{comparing distances}, $d_c(H) \leq d_{AD}(H) + 2 \leq 2\extent(F) + 2$. Let $J \cpt \mc{H}^+$ cobound a c-trivial tangle with $F'$. Then 
\[
d_c(H) \leq 2\extent(F) + 2 \leq 2\extent(J) + 2 \leq 2\tr(\mc{J}) + 2.
\]
By Lemma \ref{relations}, we have $\tr(\mc{J}) \leq \netextent(\mc{J})$ and also $\tr(\mc{J}) \leq \sqrt{\width(\mc{J})/2}$. Thus, if $\netextent(\mc{J}) = \netextent(M,\tau)$, we have
\[
d_c(H) \leq 2\netextent(\mc{H}) + 2,
\]
which contradicts our hypotheses. If, on the other hand, $\width(\mc{J}) = \width_{-2}(M,\tau)$, then 
\[
d_c(H) \leq 2\sqrt{\width(\mc{H})/2} + 2 \leq 2\netextent(\mc{H}) + 2,
\]
another contradiction.

Thus, each component of $\mc{J}^-$ must be isotopic to a component of $\mc{H}^- \cup \boundary M$. Since no two components of $\mc{J}$ are $\tau$-parallel and no component is $\boundary$-parallel, $\mc{J}$ can be isotoped so that $\mc{J}^- \subset \mc{H}^-$.
\end{proof}

Henceforth, we assume that $\mc{J}$ has been isotoped so that $\mc{J}^- \subset \mc{H}^-$ and we do not further isotope $\mc{J}^-$. We now turn to thick surfaces, where the analysis is harder. We use the ``graphic'' technology of spanning and splitting due to Johnson \cite{Johnson}. The analysis here is a combination of \cite{JT} and \cite{MHL1}. See those papers for more details.  By Lemma \ref{thin not parallel}, no two components of $\mc{J}$ are $\tau$-parallel. 

\begin{lemma}\label{Claim 1}
 Suppose that $S \cpt \mc{J}^+$ and that $F \cpt \mc{H}^- \setminus \mc{J}^-$. If $S$ and $F$ lie in the same component of $M \setminus \mc{J}^-$, then $S$ cannot be isotoped to be disjoint from $F$.
\end{lemma}
\begin{proof}
Under these circumstances, if $F$ and $S$ could be isotoped within $M'$ to be disjoint, then $F$ could be isotoped to lie in one of the c-trivial tangles on either side of $S$, but this would contradict the fact that $F$ is c-incompressible and not parallel to a component of $\mc{J}^-$.
\end{proof}

\begin{lemma}\label{Claim 2}
 We have $\mc{H}^- = \mc{J}^-$. Also, if $S \cpt \mc{J}^+$ and $H \cpt \mc{H}^+$ are in the same component of $(M,\tau)\setminus \mc{H}^-$, then either $S$ is $\tau$-parallel to $H$ in that component or $\extent(S) > \extent(H)$.
 \end{lemma}

\begin{proof}
Let $S \cpt \mc{J}^+$. Let $(W, \tau_W) \cpt (M,\tau)\setminus \mc{J}^-$ contain $S$ and suppose $H \subset W$.  Let $(U, \tau_U) \cpt (W, \tau_W) \setminus \mc{H}^-$ contain $H$. Let $\phi_H$ be a sweepout of $U$ by $H$. Let $\phi_S$ be a sweepout of $W$ by $S$. Let $\Delta \co U \to U \times W$ be the diagonal map induced by inclusion. Consider the product map $\phi = (\phi_H \times \phi_S ) \circ \Delta \co U \times W \to I \times I$. A small perturbation of $\phi$ ensures that it is a stable map. The set $\Lambda \subset I \times I$ of critical values is a closed set (called the \defn{graphic}), which is an immersed collection of curves containing vertices of degree two and four. For a value $(t, s)$ in the interior of $I \times I$, the preimage $\phi^{-1}(t,s)$ is the intersection of the c-bridge surface $H_t$ of $U$ with the c-bridge surface $S_s$ of $W$. Let $\Gamma_\pm = \phi^{-1}(\pm 1)$. Recall that it is the union of components of $\mc{J}^- \cap \boundary W$ with a graph. Since $\mc{J}^- \subset \mc{H}^-$, for all $t$, $H_t \cup \boundary W$ intersects only the graph portion of $\Gamma_\pm$. 

For $(t,s)$ a regular value in the interior of the graphic. Say $H_t$ is \defn{mostly above} $S_s$ if the intersection of $H_t \cup \boundary W$ with the c-trivial tangle below $S_s$ is contained in the union of pairwise disjoint  unpunctured and once-punctured discs in $H_t \cup \boundary W$. Define $H_t$ to be \defn{mostly below} $S_s$ in a similar fashion. Let $Q_a$ (resp. $Q_b$) be the set of points of the graphic such that $H_t$ is mostly above (resp. mostly below) $S_s$. Note that if $(t, s)$ and $(t', s')$ are in the same component of $I \times I \setminus \Lambda$, then $(t, s)$ is in $Q_a$ (resp. $Q_b$) if and only if $(t', s')$ is also. 

Fix some $t$ in the interior of $I$. When $s$ is very close to $-1$, $H_t \cup \boundary W$ intersects the spine of $\Gamma_-$ in discs that are meridian discs for some of the edges of $\Gamma_-$. On $H_t$, these discs are regular neighborhoods of the points of intersection between the edges of $\Gamma_-$ and $H_t$. Thus, for values of $s$ very close to $-1$, $(t,s) \in Q_a$. Similarly, for values of $s$ very close to $+1$, $(t, s) \in Q_b$. Thus, there exists a $s_0 \in (-1,1)$ such that one of the following holds:
\begin{itemize}
\item (Splitting) For each $(t, s_0)$, $(t, s_0)$ is neither in $Q_a$ nor in $Q_b$. Furthermore, if $y = s_0$ passes through a vertex of the graphic then the interior of the region above the vertex lies in $Q_b$ and the interior of the region lying below the vertex lies in $Q_a$.
\item (Spanning) There exists $t_- < t_+ \in I$ such that $(t_-, s_0) \in Q_a$ and $(t_+, s_0) \in Q_b$ or vice versa.
\end{itemize}

Assume we are in the Splitting case. Let $\psi \co U \to I$ be the sweepout $\phi(\cdot, s_0)$ corresponding to $H$. We claim that it is adapted to $S' = S_{s_0} \cap U$. Since $\phi$ is stable, the function $\psi' = \psi|_{S'}$ is Morse. The line $y = s_0$ passes through at most one vertex of the graphic. Observe that $\psi'$ has an exceptional critical value if and only if it passes through a valence 4 vertex of the graphic. Using the fact that $H$ is a sphere with more than four punctures, an argument nearly verbatim to that of \cite[Lemma 8]{MHL1} (which is itself modelled on \cite[Lemma 5.6]{RS}) shows that Condition (A) is satisfied. By Proposition \ref{prop: distance bd} and Lemma \ref{comparing distances}, 
\[
d_c(H) \leq d_{AD}(H) + 2 \leq 2\extent(S) + 4,
\]
contradicting our hypotheses, as in the proof of Proposition \ref{thin surface isotopy}.

Now assume that we are in the Spanning Case. We follow the argument of \cite[Theorem 3.1]{JT}. Color the c-trivial tangle below $S_{s_0}$ blue and the c-trivial tangle above $S_{s_0}$ red. Since $(t_-, s_0) \in Q_a$, the surface $H_{t_-}$ is mostly above $S_{s_0}$, each curve of $(H_{t_-} \cup \boundary W) \cap S_{s_0}$ is inessential in $(H_{t_-} \cup \boundary W)$. An innermost such component $\ell$ bounds a zero or once-punctured disc in $H_{t_-} \cup \boundary W$ that is blue. Similarly, each curve of $(H_{t_+} \cup \boundary W) \cap S_{s_0}$ is inessential in $(H_{t_-} \cup \boundary W)$ and an innermost such component bounds a red zero or once-punctured disc in $H_{t_-} \cup \boundary W$. Since a disc cannot be both blue and red, the surface $S_{s_0}$ is disjoint from $\boundary W$. By Claim 1, this implies that $U = W$. The argument \emph{verbatim} from \cite[Theorem 3.1]{JT}, shows that there is a sequence of c-compressions and isotopies of $S$ which produces a surface with a component isotopic to $H$. In particular, $\extent(S) \geq \extent(H)$ with equality if and only if $S$ and $H$ are $\tau$-parallel in $U = W$. 

Applying this argument to each $S \cpt \mc{J}^+$ proves the claim.
\end{proof}

Let $(T', \lambda')$ be the width tree associated to $\mc{J}$ and $(T, \lambda)$ the tree associated to $(M,\tau)$. By Lemma \ref{Claim 2}, $\mc{H}^- = \mc{J}^-$. Thus, the unoriented trees underlying $T$ and $T'$ are isomorphic. Again, by Lemma \ref{Claim 2}, for each vertex $v$ of $T$ (equivalently, of $T'$), we have $\lambda'(v) \geq \lambda(v)$. If equality holds, then the corresponding thick surfaces are $\tau$-parallel. Thus, $\netextent(\mc{J}) \geq \netextent(\mc{H})$ and $\width(\mc{J}) \geq \width(\mc{H})$ and if equality holds for either, then $\mc{J}$ and $\mc{H}$ are $\tau$-parallel. By hypothesis, $\mc{J}$ realizes $\netextent_{-2}(M,\tau)$ or $\width_{-2}(M,\tau)$. Consequently, $\mc{J}$ is $\tau$-parallel to $\mc{H}$. Therefore, as an undirected graph $T'$ is isomorphic to $T$ by an isomorphism $\phi$ that preserves labels. We desire to show that the orientations of $T$ and $T'$ are the same or reversed. 

The set of thick surfaces of $\mc{J}$ bounding a component of $S^3 \setminus \mc{J}$ disjoint from $\mc{J}^-$ correspond exactly to the source and sinks of $T'$. Similarly, with the thick surfaces of $\mc{H}$ and the sources and sinks of $T$. Thus, the isotopy of $\mc{J}$ to $\mc{H}$ takes the sources and sinks of $T'$ to the sources and sinks of $T$.  Suppose that $e'_0$ and $e'_1$ are edges incident to a vertex $v'$ of $T'$. Let $e_i = \phi(e'_i)$ and $v = \phi(v')$ for $i = 1,2$. Let $F'_i \cpt \mc{J}^-$ and $J \cpt \mc{J}^+$ correspond to $e'_i$ and $v'$ for $i = 0,1$. Let $F_i \cpt \mc{H}^+$ and $H \cpt \mc{H}$ correspond to $e_i$ and $v$. The isotopy takes $F'_i$ to $F_i$ and $J$ to $H$. One of $e'_0$ and $e'_1$ is incoming to $v'$ and the other outgoing if and only if $F'_0$ and $F'_1$ are on opposite sides of $J$. Likewise, one of $e_0$ and $e_1$ is incoming to $v$ and the other outgoing if and only if $F_0$ and $F_1$ are on opposite sides of $H$. The isotopy of $\mc{J}$ to $\mc{H}$ preserves the separation properties of its surfaces. Thus, one of $e'_0$ and $e'_1$ is incoming to $v'$ and the other outgoing if and only if one of $e_0$ and $e_1$ is incoming to $v$ and the other outgoing. It follows that $\phi$ either preserves all edge orientations or reverses all edge orientations.

\end{document}